\documentclass[12pt]{article}

\usepackage{amsmath}
\usepackage{amssymb}
\usepackage{fullpage}
\usepackage{amsthm}
\usepackage{graphicx}
\usepackage{placeins}
\usepackage{caption}
\usepackage{extpfeil}
\usepackage{enumerate}
\usepackage{caption}
\usepackage{subcaption}
\usepackage{hyperref}
\usepackage{xcolor}
\usepackage{pbox}
\usepackage[numbers,sort&compress]{natbib}

\newcommand{\be}{\begin{equation}}
\newcommand{\ee}{\end{equation}}
\newcommand{\bee}{\begin{equation*}}
\newcommand{\eee}{\end{equation*}}

\newcommand{\dive}{\mbox{div}}

\newtheorem{thm}{Theorem}[section]
\newtheorem{prop}{Proposition}[section]

\newtheorem{rmk}{Remark}[section]
\newtheorem{lem}{Lemma}[section]

\newtheorem{cor}{Corollary}[section]\label{key}

\begin{document}

%\title{Asymptotic stability of $(-1)$-homogeneous solutions of the $3D$ incompressible stationary  Navier-Stokes equations}
\title{Asymptotic stability of homogeneous solutions of  incompressible stationary  Navier-Stokes equations}
\author{YanYan Li\footnote{Department of Mathematics, Rutgers University, 110 Frelinghuysen Road, Piscataway, NJ 08854, USA. Email: yyli@math.rutgers.edu,   is partially supported by NSF grant DMS-1501004.}, 
Xukai Yan\footnote{School of Mathematics, Georgia Institute of Technology, 686 Cherry St NW, Atlanta, GA 30313, USA. Email: xukai.yan@math.gatech.edu,  is partially supported by AMS-Simons Travel Grant and AWM-NSF Travel Grant 1642548.}}
\date{}
%\thanks{$^1$Department of Mathematics, Rutgers University, 110 Frelinghuysen Road, Piscataway, NJ 08854. Email: yyli@math.rutgers.edu. YYL is partially supported by NSF grant DMS-1501004.}
%\thanks{$^2$School of Mathematics, Georgia Institute of Technology, 686 Cherry Street, Atlanta, GA 30332-0160 USA. E-mail: xukai.yan@math.gatech.edu. XY is partially supported by AMS-Simons Travel Grant and AWM-NSF Travel Grant 1642548.}
%to sverak: comments and suggestions are welcome. 

\maketitle

\abstract
   %The asymptotic stability of Landau solutions have been studied by \cite{Karch}. 
   {It was proved by Karch and Pilarzyc that  Landau solutions  are asymptotically stable under any $L^2$-perturbation. 
   In our earlier work with L. Li, we have classified all  $(-1)$-homogeneous axisymmetric no-swirl solutions of incompressible stationary Navier-Stokes equations in three dimension which are smooth on the unit sphere minus the south and north poles. 
   In this paper, we study the asymptotic stability of the least singular solutions among these solutions other than Landau solutions, 
   %$(-1)$-homogeneous axisymmetric no-swirl solutions with singularities at the south and north poles on the unit sphere, 
    and prove that such solutions are asymptotically stable under any $L^2$-perturbation.}

\section{Introduction}\label{sec_1}
%   We study the initial value problem of 3 dimensional incompressible Navier-Stokes equations
%\begin{equation}\label{NSE}
%\left\{
%    \begin{split}
%       & u_t-\Delta u+(u\cdot \nabla) u+\nabla p=0. %F, \quad (x,t)\in \mathbb{R}^3 \times (0,\infty)\\
%       & \dive u=0,\\
%       & u(x,0)=u_0(x),
%   \end{split}
%   \right.
%\end{equation}
%where $u$ is the velocity, $p$ is the pressure and $F$ is the force. The existence of solutions are well-known and the uniqueness and regularity of the solutions still remain open. 

%If $u(x,t)$ is a solution of (\ref{NSE}) independent of $t$, then $u$ is a stationary solution, satisfying
Consider the incompressible stationary Navier-Stokes Equations in $\mathbb{R}^3$, 
\begin{equation}\label{NS}
\left\{
    \begin{split}
       & -\Delta u+(u\cdot \nabla) u+\nabla p=0, \\ %F, \quad x\in \mathbb{R}^3\\
       & \dive\textrm{ }  u=0.
   \end{split}
   \right.
\end{equation}

These equations are invariant under the scaling $u(x)\to \lambda u(\lambda x)$ and $p(x)\to \lambda^2 p(\lambda x)$, $\lambda>0$ and it is natural to study solutions which are invariant under this scaling. These solutions are refered as  $(-1)$-homogeneous solutions (although $p$ is $(-2)$-homogeneous). %For such solutions $u$ is $(-1)$-homogeneous and $p$ is $(-2)$-homogeneous, and we refer them as $(-1)$-homogeneous solutions.

%Due to the $(-1)$-homogeneity of solutions we study, we only need to look at the solutions on the unit sphere $\mathbb{S}^2$. 
%(\ref{NS}) can be reformulated in spherical coordinates $(r,\theta,\phi)$.

Let $x=(x_1,x_2,x_3)$ be Euclidian coordinates and $e_1=(1,0,0),e_2=(0,1,0),e_3=(0,0,1)$ be the corresponding unit normal vectors. Denote $x'=(x_1,x_2)$. Let $(r,\theta, \phi)$ be the spherical coordinates, where $r$ is the radial distance from the origin, $\theta$ is the angle between the radius vector and the positive $x_3$-axis, and $\phi$ is the meridian angle about the $x_3$-axis.  A vector field $u$ can be written as
\[
%\label{u_polar}
	u = u_r e_r + u_\theta e_{\theta} + u_\phi e_{\phi},
\]
where
\[
	e_r = \left(
	\begin{matrix}
		\sin\theta\cos\phi \\
		\sin\theta\sin\phi \\
		\cos\theta
	\end{matrix} \right),  \hspace{0.5cm}
	e_{\theta} = \left(
	\begin{matrix}
		\cos\theta\cos\phi  \\
		\cos\theta\sin\phi   \\
		-\sin\theta	
	\end{matrix} \right), \hspace{0.5cm}
	e_{\phi} = \left(
	\begin{matrix}
		-\sin\phi \\  \cos\phi \\ 0
	\end{matrix} \right).
\]
A vector field $u$ is called axisymmetric if $u_r$, $u_{\theta}$ and $u_{\phi}$ are independent of $\phi$, and is called {\it no-swirl} if $u_{\phi}=0$.

 In 1944, L.D. Landau discovered a 3-parameter family of explicit $(-1)$-homogeneous
solutions of the stationary NSE in $C^\infty(\mathbb{R}^3\setminus\{0\})$. %Theses solutions satisfy 
%\[
%   -\Delta u+(u\cdot \nabla) u+\nabla p=b\delta_0\vec{a}, \quad x\in\mathbb{R}^3, 
%\] 
%across the origin %with $F=b\delta_0\vec{a}$ 
%for some constant $b$ and some unit vector $\vec{a}$. 
These solutions, now called Landau solutions, are axisymmetric with no-swirl and have exactly one singularity at the origin.  
%In the axisymmetric no-swirl case (\ref{NS}) was converted earlier to an equation of Riccati type by Slezkin in \cite{SL}. The
%Riccati type equation was later independently derived by Yatseyev using a different
%method in \cite{Y}, where various exact solutions were given. The Landau solutions were
%also independently found by Squire in \cite{SQ}.
%The solutions were also independently found by Squire in \cite{SQ}. 
Tian and Xin proved in \cite{TianXin} that all $(-1)$-homogeneous, axisymmetric nonzero solutions 
of  (\ref{NS}) in $C^\infty(\mathbb{R}^3\setminus\{0\})$
are Landau solutions.  \v{S}ver\'{a}k proved in \cite{Sverak} that all (-1)-homogeneous nonzero solutions of (\ref{NS}) in $C^\infty(\mathbb{R}^3\setminus\{0\})$ are Landau solutions. 
There have also been works on   $(-1)$-homogeneous solutions of (\ref{NS}), see  \cite{G, PP1,PP2, PP3, Serrin, SL, SQ, W, Y}. 
%There have also been works on   $(-1)$-homogeneous solutions of (\ref{NS}) with singularities at the south pole $S$ or the north pole $N$ on the unit sphere $\mathbb{S}^2$, see  \cite{G, PP1,PP2, PP3, Serrin, SL, SQ, W, Y}. 
In \cite{LLY1, LLY2, LLY3}, the $(-1)$-homogeneous axisymmetric solutions of (\ref{NS}) in $C^{\infty}(\mathbb{R}^3\setminus\{(x_1,x_2)=0\})$ with a possible singular ray $\{(x_1,x_2)=0\}$ was studied, 
%possible singularities at the ray $\{(x_1,x_2)=0\}$. 
 where such solutions with no-swirl  were classified in \cite{LLY1} and \cite{LLY2}, and  existence of such solutions with nonzero swirl was proved in \cite{LLY1} and \cite{LLY3}.

There has been much work in literature on the existence of weak solutions and $L^2$-decay of weak solutions of the evolutionary Navier-Stokes equations, see e.g.  \cite{ Cannone, Kaj, Kato, Lemari, Leray, Masuda, Schonbeck, Teman} and the references therein. Such $L^2$-decay of weak solutions can be viewed as the asymptotically stability of the zero stationary solution of (\ref{NS}). 
%As a solution of (\ref{NS}) in $\mathbb{R}^3$, $u=0$ is considered to be asymptotically stable in $L^2(\mathbb{R}^3)$, due to the $L^2$-decay of weak solutions of evolutionary Navier-Stokes equations in time. 
The asymptotic stability problem has been studied for other nonzero stationary solutions of (\ref{NS}) with some possible singularities in $\mathbb{R}^3$. Karch and Pilarzyc proved in \cite{Karch} that small Landau solutions are asymptotically stable under $L^2$-perturbations.  The $L^2$ asymptotic stability of other solutions with singularities are also studied in \cite{KPS}. %(For background about asymptotic stability of stationary solutions please refer to \cite{Karch} and \cite{KPS}.)  
With special $(-1)$-homogeneous solutions which are different from Landau solutions obtained in \cite{LLY1,LLY2, LLY3}, it is worth to explore the asymptotic stability or instability of these solutions. In this paper, we start this study for a family of solutions which are the simplest and least singular solutions among the solutions found in \cite{LLY1,LLY2, LLY3}. 

%In \cite{LLY1}, \cite{LLY2} and \cite{LLY3}, we studied $(-1)$-homogeneous axisymmetric solutions of (\ref{NS}) in $C^{\infty}(\mathbb{R}^3\setminus\{(x_1,x_2)=0\}$ with possible singularities at the ray $\{(x_1,x_2)=0\}$. We have classified such solutions with no-swirl in \cite{LLY1} and \cite{LLY2}, and proved the existence of such solutions with nonzero swirl in \cite{LLY1} and \cite{LLY3}. 

%\marginpar{Introduce polar coordinates and $u_{r}, u_{\theta}$}

%We briefly introduce the family of solutions we have obtained. 
%Now let us introduce this family of solutions. 
Denote $U= u\cdot r \sin\theta$ and $y=\cos\theta$. By the divergence free property of $u$ we have $u_{r}=\frac{1}{r}U_{\theta}'$. For (-1)-homogeneous axisymmetric no-swirl solutions, (\ref{NS}) can be reduced to
\begin{equation}\label{eqNS_1}
   (1-y^2)U'_{\theta}+2yU_{\theta}+\frac{1}{2}U^2_{\theta}=c_1(1-y)+c_2(1+y)+c_3(1-y^2).
\end{equation}
%\textbf{Case 1}. $c_1=c_2=0$, $U_{\theta}(\pm 1)=0$
For $c_1\ge -1$ and $c_2\ge -1$, let 
\[
   \bar{c}_3 (c_1,c_2) := -\frac{1}{2} \left( \sqrt{1+c_1} + \sqrt{1+c_2}  \right) \left( \sqrt{1+c_1} + \sqrt{1+c_2}  + 2 \right),
   \]
   where $c_1, c_2, c_3$ are real numbers.  
    Denote $c=(c_1,c_2,c_3)$ and 
\[
   J:=\{c\in\mathbb{R}^3\mid c_1\ge -1, c_2\ge -1, c_3\ge \bar{c}_3(c_1,c_2)\}.
\]
In \cite{LLY2}, it was proved that there exist  $\gamma^-,\gamma^+\in C^0(J, \mathbb{R})$, satisfying $\gamma^-(c)<\gamma^+(c)$ if $c_3>\bar{c}_3(c_1,c_2)$, and $\gamma^-(c)=\gamma^+(c)$ if $c_3=\bar{c}_3(c_1,c_2)$, such that equation (\ref{eqNS_1}) has a unique solution $U_{\theta}^{c,\gamma}$ in $C^{\infty}(-1,1)\cap C^0[-1,1]$ satisfying $U_{\theta}^{c,\gamma}(0)=\gamma$ for every $c$ in $J$ and $\gamma^-(c)\le \gamma \le \gamma^+(c)$.  
In particular, $\gamma^+(0)>0$ and $\gamma^-(0)<0$. 
 %For any $c$ in $J$ and $\gamma^-(c)\le \gamma \le \gamma^+(c)$, (\ref{eqNS_1}) has a unique solution $U^{c,\gamma}$ in $C^{\infty}(-1,1)\cap C^0[-1,1]$ satisfying $U^{c,\gamma}(0)=\gamma$. 
 Moreover, let 
\begin{equation}\label{eq_u_cgamma}
  \begin{split}
    & u^{c,\gamma}\equiv u^{c,\gamma}_re_r+u^{c,\gamma}_{\theta}e_{\theta}=(U^{c,\gamma}_{\theta})'e_r+\frac{U^{c,\gamma}_{\theta}}{\sin\theta}e_{\theta}, \\
    &  p^{c,\gamma}=\frac{1}{r^2}(u^{c,\gamma}_r-\frac{1}{2}(u^{c,\gamma}_{\theta})^2)=\frac{1}{r^2}((U^{c,\gamma}_{\theta})'-\frac{(U^{c,\gamma}_{\theta})^2}{2\sin^2\theta}).
    \end{split}
\end{equation}
 $\{(u^{c,\gamma}, p^{c,\gamma})\mid c\in J, \gamma^-(c)\le \gamma\le \gamma^+(c)\}$ are all $(-1)$-homogeneous axisymmetric no-swirl solutions of (\ref{NS}) in $C^{\infty}(\mathbb{R}^3\setminus\{(x_1,x_2)=0\})$.
%such that  all $(-1)$-homogeneous axisymmetric no-swirl solutions of \eqref{NS} in $C^2(\mathbb{R}^3\setminus\{(x_1,x_2)=0\}$ are a four parameter family $\{(u^{c,\gamma},p^{c,\gamma})\}$, satisfying (\ref{eqNS_1}) and $U^{c,\gamma}_{\theta}(0)=\gamma$, with $(c,\gamma)\in I$  where
%
%
%
%In \cite{LLY2}, we have proved that for each $c$ in $J$,  there exist $\gamma^-(c),\gamma^+(c)\in C^0(J, \mathbb{R})$ such that  all $(-1)$-homogeneous axisymmetric no-swirl solutions of \eqref{NS} in $C^2(\mathbb{R}^3\setminus\{(x_1,x_2)=0\}$ are a four parameter family $\{(u^{c,\gamma},p^{c,\gamma})\}$, satisfying (\ref{eqNS_1}) and $U^{c,\gamma}_{\theta}(0)=\gamma$, with $(c,\gamma)\in I$  where
%\begin{equation*}%\label{eq:c1c2:I}
%	I:= \{(c,\gamma)\in \mathbb{R}^4 \mid c_1\geq -1, c_2\geq -1, c_3\geq \bar{c}_3(c_1,c_2), \gamma^-(c) \leq \gamma\leq \gamma^+(c) \}. 
%\end{equation*}
%\[
%   I:=\{(c,\gamma)\mid c_1\ge -1, c_2\ge -1, c_3\ge \bar{c}_3, \gamma^-(c)\le \gamma\le \gamma^+(c)\}.
%\]
%Moreover, by Theorem 1.3 in \cite{LLY2}, we have 
It was also obtained in \cite{LLY2} that
\begin{equation*}%\label{eqnoswirl_1}%\label{sec2:eq:tau}
\begin{split}
		& U_\theta^{c,\gamma}(-1) = \left\{
		\begin{array}{ll}
			2+2\sqrt{1+c_1}, &  \mbox{when } \gamma=\gamma^+(c), \\
			2-2\sqrt{1+c_1}, &  \mbox{otherwise},
		\end{array}
		\right. \\
               & 		U_\theta^{c,\gamma}(1) = \left\{
		\begin{array}{ll}
			-2-2\sqrt{1+c_2}, &  \mbox{when } \gamma=\gamma^-(c), \\
			-2+2\sqrt{1+c_2}, &  \mbox{otherwise}.
		\end{array}
		\right.
		\end{split}
	\end{equation*}

As mentioned earlier, we would like to study the asymptotic stability or instability of the  $(-1)$-homogeneous axisymmetric stationary solutions found in \cite{LLY1, LLY2, LLY3}. %For different solutions with different singular behaviors, we may need different methods to study their stability. 
Different from Landau solutions, these solutions %we have obtained include families of solutions which 
  are singular at the north pole $N$ and/or south pole $S$. %, with different singular behaviors for different solutions. 
 These solutions $u$ satisfy either %$\displaystyle\liminf_{|x|=1, x'\to 0}|x||x'||\nabla u(x)|>0$, 
 $0<\displaystyle\limsup_{|x|=1,x'\to 0}|x||x'||\nabla u(x)|<\infty$ or $\displaystyle\limsup_{|x|=1,x'\to 0}|x'|^2|\nabla u(x)|>0$, while  Landau solutions  satisfy $\displaystyle \sup_{|x|=1}|x|^2|\nabla u|<\infty$. %$\displaystyle \limsup_{x\to 0}|x|^2|\nabla u(x)|<\infty$ 
% but $\displaystyle\limsup_{|x|=1,x'\to 0}|x||x'||\nabla u(x)|=0$. 
 In this paper, we study the stability of $(-1)$-homogeneous axisymmetric no-swirl solutions satisfying 
$0<\displaystyle\limsup_{x'\to 0}|x||x'||\nabla u(x)|<\infty$. These solutions are the family $\{(u^{c,\gamma}, p^{c,\gamma})\mid (c,\gamma)\in M\}$, where 
\begin{equation}\label{eqS_M}
  M:=\{(c,\gamma)\mid c_1=c_2=0, c_3>-4, \gamma^-(c)<\gamma<\gamma^+(c)\}.
\end{equation}

For any $(c,\gamma)\in M$, $U_{\theta}^{c,\gamma}$ satisfies 
\begin{equation}\label{eqS_1}
\left\{
    \begin{split}
    & (1-y^2)(U_{\theta}^{c,\gamma})'+2yU_{\theta}^{c,\gamma}+\frac{1}{2}(U_{\theta}^{c,\gamma})^2=c_3(1-y^2), -1<y<1, \\
    &U_{\theta}(0)=\gamma.
    %& U_{\theta}(1)=U_{\theta}(-1)=0
   \end{split}
   \right.
\end{equation}
%state: $\gamma^+(0)>0$ and $\gamma^-(0)<0$. So there is some $\mu_0$, such that $\{|(x,\gamma)|\le \mu_0\}\subset M$.

%where $c=(0,0,c_3)$,  $c_3>\bar{c}_3(0, 0)=-4$ and $\gamma^-(c)<\gamma<\gamma^+(c)$. Then by (\ref{eqnoswirl_1}), %Theorem 1.3 in \cite{LLY2}, 
%$U_{\theta}^{c,\gamma}(\pm 1)=0$.  

%Denote the solution of (\ref{eqS_1}) by $U_{\theta}^{c,\gamma}$, then let
%\begin{equation}\label{eq_u_cgamma}
%  \begin{split}
%    & u^{c,\gamma}=(U^{c,\gamma}_{\theta})'e_r+\frac{U^{c,\gamma}_{\theta}}{\sin\theta}e_{\theta}, \\
%    &  p^{c,\gamma}=\frac{1}{r^2}(u^{c,\gamma}_r-\frac{1}{2}(u^{c,\gamma}_{\theta})^2)=\frac{1}{r^2}(U_{\theta}'-\frac{U_{\theta}^2}{2\sin^2\theta}).
%    \end{split}
%\end{equation}.
%We have $(u^{c,\gamma}, p^{c,\gamma})$ is a solution of (\ref{NS}) in $\mathbb{R}^3\setminus\{(x_1,x_2)=0\}$. 

%Let 
%\begin{equation}\label{eqS_M}
%  M:=\{(c,\gamma)\mid c_1=c_2=0, c_3>-4, \gamma^-(c)<\gamma<\gamma^+(c)\}.
%\end{equation}
%We study the 
%Denote $e_3=(0,0,1)$. 
\begin{prop}\label{prop_force}
   Let $(c,\gamma)\in M$, then $(u^{c,\gamma}(x), p^{c,\gamma}(x))$ satisfies 
   \begin{equation}\label{eqF_1}
   \left\{
     \begin{split}
       & -\Delta u^{c,\gamma}+u^{c,\gamma}\cdot\nabla u^{c,\gamma}+\nabla p^{c,\gamma}=(4\pi c_3 \ln |x_3| \partial_{x_3}\delta_{(0,0,x_3)}-b^{c,\gamma}\delta_{0})e_3, \quad x\in \mathbb{R}^3, \\
       & \dive\textrm{ }u^{c,\gamma}=0, \quad x\in \mathbb{R}^3,
      \end{split}
      \right.
   \end{equation}
   where 
   \begin{equation}\label{eqA_b}
       b^{c,\gamma}=\int_{-1}^{1}\left(y|U_{\theta}'|^2-\frac{2-y^2}{1-y^2}U_{\theta}-\frac{y}{1-y^2}U_{\theta}^2 \right)dy.%\int_{-1}^{1}\left(y|(U^{c,\gamma}_{\theta})'|^2-\frac{5}{2}U^{c,\gamma}_{\theta}-\frac{y}{2(1-y^2)}(U^{c,\gamma}_{\theta})^2 \right)dy.
         \end{equation} 
         Equation (\ref{eqF_1}) and (\ref{eqA_b}) are understood in the following distribution sense: for any $\varphi\in C_c^{\infty}(\mathbb{R}^3)$, $j=1,2,3$, 
         \begin{equation}\label{eqF_2}
           \int_{\mathbb{R}^3} (\nabla u_j \nabla \varphi-u_iu_j\partial_{x_i}\varphi-p\partial_{x_j}\varphi)= [4\pi c_3\int_{-\infty}^{\infty}\ln|x_3|\partial_{x_3}\varphi(0,0,x_3)dx_3-b^{c,\gamma}\varphi(0)]\delta_{j3}e_3, 
         \end{equation}
         and 
         \begin{equation}\label{eqF_3}
            \int_{\mathbb{R}^3}u^{c,\gamma}\cdot\nabla \varphi=0.
         \end{equation}
      % and $u^{c,\gamma}$ satisfies (\ref{eqF_1}) in distributional sense that for any $\varphi(x)\in C_c^{\infty}(\mathbb{R}^3)$, 
%       \[
%          \int_{\Omega}T_{ij}\partial_{x_i}\varphi dx=(4\pi c_3\int_{-R}^{R}\ln |x_3|\partial_{x_3}\varphi(0,0,x_3)dx_3-b^{c,\gamma}\varphi(0))e_3. 
%       \]
\end{prop}
%\begin{prop}\label{prop_force}
%   Let $c=(0, 0, c_3)$, $c_3>-4$, and $\gamma^-(c)<\gamma<\gamma^+(c)$, then $(u^{c,\gamma}(x), p^{c,\gamma}(x))$ satisfies 
%   \[
%      -\Delta u^{c,\gamma}+u^{c,\gamma}\cdot\nabla u^{c,\gamma}+\nabla p^{c,\gamma}=(4\pi c_3 \ln |x_3| \partial_{x_3}\delta_{(0,0,x_3)}-b^{c,\gamma}\delta_{0})e_3, \quad x\in \mathbb{R}^3, 
%   \]
%   where 
%   \begin{equation}\label{eqA_b}
%       b^{c,\gamma}=\int_{-1}^{1}\left(y|(U^{c,\gamma}_{\theta})'|^2-2U^{c,\gamma}_{\theta}-\frac{y}{2(1-y^2)}(U^{c,\gamma}_{\theta})^2 \right)dy.
%   \end{equation}
%\end{prop}
% Then by (\ref{eqF}) in Section \ref{sec2_2}, $u^{c,\gamma}=(U^{c,\gamma}_{\theta})'e_r+\frac{U^{c,\gamma}_{\theta}}{\sin\theta}e_{\theta}$ is a solution of (\ref{NS}) with the force 
%\begin{equation}\label{eqF_1}
%   F(c,\gamma)=4\pi c_3 \ln |x_3| \partial_{x_3}\delta_{(0,0,x_3)}+b\delta_{0},
%\end{equation}
%for some constant $b=b(c,\gamma)$.

We now study the stability of the family of solutions $\{u^{c,\gamma}\mid (c,\gamma)\in M\}$. 
%Let $(u^{c,\gamma}, p^{c,\gamma})$ to be the solution of (\ref{NS}) obtained by $u^{c,\gamma}_r=\frac{1}{r}(U^{c,\gamma}_{\theta})'(x)$, $u^{c,\gamma}_{\theta}=\frac{U^{c,\gamma}_{\theta}}{r\sin\theta}$, $u^{c,\gamma}_{\phi}=0$. We study the stability of the family of solutions $\{u^{c,\gamma}\}$. 
Let $\dot{H}^1(\mathbb{R}^3)$ denote the closure of $C^{\infty}_c(\mathbb{R}^3, \mathbb{R}^3)$ under the norm $\|\nabla u\|_{L^2(\mathbb{R}^3)}$, and for $1\le p <\infty$, 
\[
  %\begin{split}
  %& \dot{H}^1(\mathbb{R}^3)=\{u\in L^1_{loc}(\mathbb{R}^3)\mid \nabla u\in L^2(\mathbb{R}^3)\},\\
   L^p_{\sigma}(\mathbb{R}^3)=\{u\in L^p(\mathbb{R}^3) \mid \dive \textrm{ }u=0\}, \quad 
   \dot{H}^1_{\sigma}(\mathbb{R}^3)=\{u\in \dot{H}^1(\mathbb{R}^3)\mid \dive \textrm{ } u=0\},
%   \end{split}
\]
and
\[
   \|u\|_{L^p_{\sigma}(\mathbb{R}^3)}:= \|u\|_{L^p(\mathbb{R}^3)}, \quad \|u\|_{\dot{H}^1_{\sigma}(\mathbb{R}^3)}=\|\nabla u\|_{L^2(\mathbb{R}^3)}. 
\]

For a given solution $(u^{c,\gamma},p^{c,\gamma})$ of  (\ref{NS}), let $u=u(x,t)$ denote a solution of 
\begin{equation}\label{NSE}
   \left\{
    \begin{split}
       & u_t-\Delta u+(u\cdot \nabla) u+\nabla p=(4\pi c_3 \ln |x_3| \partial_{x_3}\delta_{(0,0,x_3)}-b^{c,\gamma}\delta_{0})e_3, \textrm{ } (x,t)\in \mathbb{R}^3 \times (0,\infty), \\
       & \dive\textrm{ } u=0, \textrm{ } (x,t)\in \mathbb{R}^3 \times (0,\infty),\\
       & u(x,0)=u^{c,\gamma}+w_0,
   \end{split}
   \right.
\end{equation}
where $w_0\in L^2_{\sigma}(\mathbb{R}^3)$ and $b^{c,\gamma}$ is given by (\ref{eqA_b}). Then $w(x,t)=u(x,t)-u^{c,\gamma}$ and $\pi(x)=p(x)-p^{c,\gamma}(x)$ satisfy the initial value problem 

%For a given solution $(u^{c,\gamma},p^{c,\gamma})$ of  (\ref{NS}) satisfying (\ref{eqS_1}), denote $u=u(x,t)$ as a solution of (\ref{NSE}) in $(\mathbb{R}^3\setminus\{(x_1,x_2)=0\}\times (0, +\infty)$,  %with $F=F(c,\gamma)$, where $F(c,\gamma)$ is given by (\ref{eqF_1}), 
 %and the initial value $u_0=u^{c,\gamma}+w_0$, where $w_0\in L^2_{\sigma}(\mathbb{R}^3)$. Then the functions $w(x,t)=u(x,t)-u^{c,\gamma}$ and $\pi(x)=p(x)-p^{c,\gamma}(x)$ satisfy the initial value problem 
\begin{equation}\label{eqE_1}
   \left\{
    \begin{split}
       & w_t-\Delta w+(w\cdot \nabla) w+(w\cdot \nabla)u^{c,\gamma}+(u^{c,\gamma}\cdot \nabla) w+\nabla \pi=0, \textrm{ } (x,t)\in \mathbb{R}^3 \times (0,\infty),\\ %\quad (x,t)\in \mathbb{R}^3\times (0,\infty)\\
       & \dive\textrm{ }  w=0,\textrm{ } (x,t)\in \mathbb{R}^3 \times (0,\infty),\\
       & w(x,0)=w_0(x).
   \end{split}
   \right.
\end{equation}
We study the existence and asymptotic behavior of global-in-time weak solutions of (\ref{eqE_1}). Let the energy space
  \begin{equation*}
  X:=L^{\infty}([0, \infty), L^2_{\sigma}(\mathbb{R}^3))\cap L^2([0,\infty), \dot{H}^1_{\sigma}),
   %X_T:=L^{\infty}_w([0,T], L^2_{\sigma}(\mathbb{R}^3))\cap L^2([0,T], \dot{H}^1_{\sigma}(\mathbb{R}^3)).
\end{equation*}
and for $w$ in $X$
\begin{equation*}
   \|w\|_{X}:=\|w\|_{L^{\infty}([0,\infty), L^2_{\sigma}(\mathbb{R}^3))}+\|w\|_{L^2([0,\infty), \dot{H}^1_{\sigma})}. 
\end{equation*}
Let $(\cdot, \cdot)$ denote the $L^2$-inner product, i.e. $(f,g)=\int_{\mathbb{R}^3}fgdx$. 
A vector  $w\in X$ is a weak solution of (\ref{eqE_1}) if  for any $0\le s\le t<\infty$ and $\varphi\in C([0,\infty), H^1_{\sigma}(\mathbb{R}^3)\cap C^1([0,\infty), L^2_{\sigma}(\mathbb{R}^3) )$, 
\begin{equation*}%\label{eqE_3}
  \begin{split}
   & (w(t), \varphi(t))+\int_{s}^{t}[(\nabla w, \nabla \varphi)+(w\cdot \nabla w, \varphi)+(w\cdot \nabla u^{c,\gamma}, \varphi)+(u^{c,\gamma}\cdot \nabla w, \varphi)]d\tau\\
   & =(w(s), \varphi(s))+\int_{s}^{t}(w,\varphi_{\tau})d\tau.
   \end{split}
\end{equation*}

%Let $K(c,\gamma)$ be the constant in Corollary \ref{corH_1}. By Corollary \ref{corH_1}, there exists some $\mu_0>0$, such that 
%\begin{equation}\label{eqE_2}
%   K(c,\gamma)<1, \textrm{ for any }|(c,\gamma)|< \mu_0.
%   \end{equation}

\begin{thm}\label{thmE}
   There exists some $\mu_0>0$, such that for any $c=(0,0,c_3)$,  %$(c,\gamma)\in M$ and 
    $|(c,\gamma)|<\mu_0$, $w_0\in L^2_{\sigma}(\mathbb{R}^3)$, there exists a weak solution  $w$ of  (\ref{eqE_1})  in the energy space $X$. 
%   \begin{equation}
%   X_T:=L^{\infty}_w([0,T], L^2_{\sigma}(\mathbb{R}^3))\cap L^2([0,T], \dot{H}^1_{\sigma}(\mathbb{R}^3)).
%\end{equation}
Moreover, $w$ is weakly continuous from $[0,\infty)$ to $L^2_{\sigma}(\mathbb{R}^3)$, and satisfies that for all $0\le s\le t<\infty$,  %the strong energy inequality
 \begin{equation}\label{eqE_4}
     \|w(t)\|_2^2+\int_{s}^{t}\|\nabla \otimes w(\tau)\|_2^2 d\tau\le \|w(s)\|_2^2. 
   \end{equation}
%   \begin{equation}\label{eqE_4}
%     \|w(t)\|_2^2+2(1-K(c,\gamma))\int_{s}^{t}\|\nabla \otimes w(\tau)\|_2^2 d\tau\le \|w(s)\|_2^2,
%   \end{equation}
%   for some $K(c,\gamma)<1$. %, almost all $s\ge 0$, including $s=0$ and all $t\ge s$.
\end{thm}
Recall that $\gamma^+(0)>0$ and $\gamma^-(0)<0$. So there is some $\mu_0'$, such that $\{(c,\gamma)\mid c_1=c_2=0, |(c_3,\gamma)|\le \mu_0'\}\subset M$. %Thus in the above theorem, for $\mu_0\le \mu_0'$, $|(c,\gamma)|\le \mu_0$ indicate $(c,\gamma)\in M$. 
%\begin{thm}\label{thmE}
%   There exists some $\mu_0>0$, such that for any $|(c,\gamma)|<\mu_0$, $w_0\in L^2_{\sigma}(\mathbb{R}^3)$, and $T>0$, there exists a weak solution  $u$ of  (\ref{eqE_1})  in the energy space $X_T$,  
%%   \begin{equation}
%%   X_T:=L^{\infty}_w([0,T], L^2_{\sigma}(\mathbb{R}^3))\cap L^2([0,T], \dot{H}^1_{\sigma}(\mathbb{R}^3)).
%%\end{equation}
% satisfying the strong energy inequality
%   \begin{equation}\label{eqE_4}
%     \|w(t)\|_2^2+2(1-K(c,\gamma))\int_{s}^{t}\|\nabla \otimes w(\tau)\|_2^2 d\tau\le \|w(s)\|_2^2,
%   \end{equation}
%   for some $K(c,\gamma)<1$, almost all $s\ge 0$, including $s=0$ and all $t\ge s$.
%\end{thm}
We also have
\begin{thm}\label{thmW}
There exists some $\mu_0>0$, such that for any $c=(0,0,c_3)$, $|(c,\gamma)|<\mu_0$ and weak solution $w\in X$ of (\ref{eqE_1}) satisfying (\ref{eqE_4}), 
   %Let $\mu_0$ be the number in Theorem \ref{thmE}, %what about change to: there exists some $\mu_0>0$, such that for any 
   % $c=(0,0,c_3)$, $|(c,\gamma)|<\mu_0$, $w\in X$ be a weak solution of (\ref{eqE_1}) satisfying (\ref{eqE_4}), then 
   \[
      \lim_{t\to \infty}\|w(t)\|_2=0.
   \]
   Moreover, if $w_0\in L^p(\mathbb{R}^3)\cap L^2_{\sigma}(\mathbb{R}^3)$ for some $\frac{6}{5}<p<2$, then there exists some constant $C>0$, depending only on $(c,\gamma)$, $n, p$ and $\|w_0\|_{p}$, such that $\|w(t)\|_2\le Ct^{-\frac{3}{2}(\frac{1}{p}-\frac{1}{2})}$, for all $t>0$.
%   \begin{equation*}%\label{eqcorW_0}
%       \|w(t)\|_2\le Ct^{-\frac{3}{2}(\frac{1}{p}-\frac{1}{2})}, \quad \forall t>0.
%   \end{equation*} 
\end{thm}

%Similar to \cite{Karch}, we have the following corollary.
%
%\begin{cor}\label{corW}
%   Let  $c=(0,0,c_3)$, $|(c,\gamma)|<\mu_0$, $w_0\in L^p(\mathbb{R}^3)\cap L^2_{\sigma}(\mathbb{R}^3)$ for some $\frac{6}{5}<p<2$, and $w\in X$ is a weak solution of (\ref{eqE_1}) satisfying (\ref{eqE_4}). Then there exists some constant $C>0$, such that
%   \begin{equation}\label{eqcorW_0}
%       \|w(t)\|_2\le Ct^{-\frac{3}{2}(\frac{1}{p}-\frac{1}{2})}, \quad \forall t>0.
%   \end{equation} 
%\end{cor}

%The proof of Theorem \ref{thmE} is by a standard Galerkin method and Theorem \ref{thmW} is proved by a semigroup theory argument, which 
Theorem \ref{thmE} and Theorem \ref{thmW} can be established using the same arguments as \cite{Karch}, as long as the special stationary solutions $u^{c,\gamma}$ satisfy the following condition 
%By \cite{Karch},Theorem \ref{thmE}, Theorem \ref{thmW} and Corollary \ref{corW} are established as long as $u^{c,\gamma}$ satisfies 
\begin{equation}\label{eqEst_1}
   |\int_{\mathbb{R}^3}(v\cdot\nabla u^{c,\gamma})\cdot w dx| \le K\|\nabla w\|_{L^2}\|\nabla v\|_{L^2},
\end{equation}
%\int_{\mathbb{R}^3}|w|^2|\nabla u^{c,\gamma}| dx+\int_{\mathbb{R}^3}|w|^2|u^{c,\gamma}|^2dx \le C(|c|+|\gamma|)\|\nabla w\|_{L^2}^2.
for some constant $K$ small enough, for any divergence free $v, w\in C_c^{\infty}(\mathbb{R}^3)$. 
%The proof of Theorem \ref{thmE} is by a standard Galerkin method and Theorem \ref{thmW} is proved by a semigroup theory argument, which follows the
 In \cite{Karch},  %the existence of weak solution $w$ is by a standard Galerkin method and the $L^2$-decay of $w$ is proved by a semigroup theory argument, where 
  (\ref{eqEst_1}) is proved by Hardy's inequality when $u^{c,\gamma}$ is replaced by small Landau solutions. 
 In this paper, we analyze the solutions $u^{c,\gamma}$ where $(c,\gamma)\in M$, and obtain $|\nabla u^{c,\gamma}|\le C(|c|+|\gamma|)/(|x||x'|)$. So (\ref{eqEst_1}) is true if we have
\begin{equation}\label{eqEst_2}
   \int_{\mathbb{R}^3}\frac{|v|^2}{|x||x'|} dx \le K\|\nabla v\|_{L^2}^2,
\end{equation}
for any $v\in C_c^{\infty}(\mathbb{R}^3)$. Notice (\ref{eqEst_2}) cannot be proved by the classical Hardy's inequality. In Section \ref{sec_H}, we prove the following extended Hardy-type inequality, which includes (\ref{eqEst_2}). 
%See also \cite{KPS} for another approach where the stability of a more general class of solutions are studied. 

%The proof of Theorem \ref{thmE} is by a standard Galerkin method and Theorem \ref{thmW} is proved by a semigroup theory argument, based on the following Hardy-type inequality. 

\begin{thm}\label{thmH}
   Let $n\ge 2$, $1\le p<n$, $u\in C_0^1(\mathbb{R}^n)$, $\alpha p>1-n$, $(\alpha+\beta)p>-n$, then there exists some constant  $C$, depending on $p$, $\alpha$ and $\beta$, such that
     \begin{equation}\label{eqH_1_1}
    \||x|^{\beta}|x'|^{\alpha} u\|_{L^p(\mathbb{R}^n)}\le C\| |x|^{\beta+\alpha-\alpha'}|x'|^{\alpha'+1}\nabla u\|_{L^p(\mathbb{R}^n)},
   \end{equation}
   for all  $\alpha' \le \alpha$. Moreover, for any $\alpha'>\alpha$ and any $C>0$, (\ref{eqH_1_1}) fails in general.
 %  if and only if $\alpha' \le \alpha$, 
  % where $x'=(x_1,x_2,...,x_{n-1})$.
\end{thm}
%In the proof of (\ref{eqEst_2}), we use (\ref{eqH_1}) with $p=2$, $\alpha=\alpha'=\beta=-\frac{1}{2}$. 
Estimate (\ref{eqEst_2}) is the special case of (\ref{eqH_1_1}) with $p=2$, $\alpha=\alpha'=\beta=-\frac{1}{2}$. Then we also have (\ref{eqEst_1}). Given (\ref{eqEst_1}), Theorem \ref{thmE} and Theorem \ref{thmW}  can be proved by the same arguments used in \cite{Karch}, see also \cite{KPS}. So in this paper   we will only  prove Theorem \ref{thmH} and (\ref{eqEst_1}). 

%Once (\ref{eqW_1_2}) is established, Theorem \ref{thmE}, Theorem \ref{thmW} and Theorem \ref{CorW} can be proved by the same argument used in \cite{Karch}. So in the following section, we will only prove Theorem \ref{thmH}. 

\begin{rmk}
In \cite{KPS},  Karch, Pilarczyk and Schonbek proved the asymptotic stability of a class of general time-dependent solutions $u$ of (\ref{NSE}) using Fourier analysis,  where  (\ref{eqEst_1})  with $u^{c,\gamma}$ replaced by $u$ is an essential assumption. A list of spaces were given in \cite{KPS} where  (\ref{eqEst_1}) is true if $u^{c,\gamma}$ is in one of those spaces. But the solutions $u^{c,\gamma}$ we discuss here are not in those spaces. 
\end{rmk}

We will analyze in Section \ref{sec_E} the singular behaviors of $u^{c,\gamma}$, $(c,\gamma)\in M$.  In Section \ref{sec_F} we  study the force of $u^{c,\gamma}$, $(c,\gamma)\in M$. Theorem \ref{thmH} will be proved in Section \ref{sec_H}. Then as stated above, Theorem \ref{thmE} and Theorem \ref{thmW} follow with the same arguments as in \cite{Karch}. %We will prove Theorem \ref{thmH} in Section \ref{sec_2}. %Then the proof of Theorem \ref{thmE}, Theorem \ref{thmW} and Corollary \ref{corW} follows from arguments similar to \cite{Karch}. For convenience we display the outline of this procedure in section \ref{sec_3}.

%The proof of the results is base on a semigroup argument, with a crucial Hardy-type inequality and the estimate of $u^{c,\gamma}$. We will prove the estimate of $u^{c,\gamma}$ and $\nabla u^{c,\gamma}$, and a Hardy-type inequality in Section \ref{sec_2}. Then the proof of Theorem \ref{thmE}, Theorem \ref{thmW} and Corollary \ref{corW} follows from arguments similar to \cite{Karch}. For convenience we display the details of this procedure in section \ref{sec_3} and \ref{sec_4}.

\textbf{Acknowlegement}: We thank Vladim\'{i}r \v{S}ver\'{a}k for bringing to our attention the work \cite{Karch} of Karch and Pilarczyk.

%\section{Preliminary}\label{sec_2}

%\subsection{Estimate of the special solutions $u^{c,\gamma}$}\label{sec2_1}
\section{Estimate of the special solutions $u^{c,\gamma}$}\label{sec_E}
%Let 
%\begin{equation}\label{eqS_M}
%  M:=\{(c,\gamma)\mid c_1=c_2=0, c_3>-4, \gamma^-(c)<\gamma<\gamma^+(c)\}.
%\end{equation}

%In this section, let $O(1)$ present some function of $-1\le y\le 1$ satisfying $|O(1)|\le C$ for some universal constant $C$, which is independent of 

\begin{lem}\label{lemS_0}
   Let $K$ be a compact subset of $M$. Then there exists some positive constant $C$, depending only on $K$, such that for any $(c,\gamma)$ in $K$ and $-1\le y\le 1$, 
    \begin{equation}\label{eqS_0_0_1}
         U^{c,\gamma}_{\theta}(y) =-\frac{c_3}{2}\mathrm{sgn}(y)(1-y^2)\ln(1-y^2)+O(1)(|c|+|\gamma|) (1-y^2),
             \end{equation}
%   \begin{equation}\label{eqS_0_0_1}
%         U^{c,\gamma}_{\theta}(y) =\frac{c_3}{2}(1-y^2)\ln(1-y^2)+O(1)(|c|+|\gamma|) (1-y^2),    \end{equation}
   \begin{equation}\label{eqS_0_0_2}
       (U^{c,\gamma}_{\theta})'(y)=c_3\ln(1-y^2)+O(1)(|c|+|\gamma|),%c_3\ln(1+y)+c_3\ln(1-y)+O(1)(|c|+|\gamma|),
   \end{equation}
   and
   \begin{equation}\label{eqS_0_0_3}
      (U^{c,\gamma}_{\theta})''(y)=-\frac{2c_3y}{1-y^2}+O(1)(|c|+|\gamma|)(|\ln(1-y^2)|^2),%=\frac{2c_3}{1-y^2}+O(1)(|c|+|\gamma|)(|\ln(1+y)|^2+|\ln(1-y)|^2),
   \end{equation}
   where $O(1)$ denotes some quantity  %a function of $y$ 
    satisfying $|O(1)|\le C$ for some positive constant $C$ depending only on $K$.
%   \[
%   \begin{split}
%&  ||\frac{U^{c,\gamma}_{\theta}}{(1-x^2)(|\ln(1+x)|+|\ln(1-x)|+1)}||_{L^{\infty}(-1,1)}+ ||\frac{(U^{c,\gamma}_{\theta})'}{|\ln(1+x)|+|\ln(1-x)|+1}||_{L^{\infty}(-1,1)}\\
%& + ||(1-x^2)(U^{c,\gamma}_{\theta})''||_{L^{\infty}(-1,1)}\le C(|c|+|\gamma|). 
% \end{split}
 %  \]
\end{lem}
\begin{proof}
   %We only prove the estimates for $-1<x\le 0$. The estimates on $[0,1)$ is similar. 
   For convenience, let $C$ be a  constant depending only on $K$, $O(1)$ be a function satisfying $|O(1)|\le C$ for all $-1\le y\le 1$,  and $C$ and $O(1)$ may vary from line to line.  %
   It is easy to see from (\ref{eqS_1}) that $U^{(0,0,c_3), \gamma}_{\theta}(y)$ and $-U^{(0,0,c_3), -\gamma}_{\theta}(-y)$ satisfy the same equation and have the same value at $y=0$ and therefore they are identically the same. So we only need to prove (\ref{eqS_0_0_1})-(\ref{eqS_0_0_3}) for $-1<y\le 0$. 
   %By Lemma 2.2 in \cite{LLY2}, there exists some constant $C$, such that $\|U^{c,\gamma}_{\theta}\|_{L^{\infty}(-1,1)}\le C$. 

  By Theorem 1.5 in \cite{LLY2}, there exists some constant $C$, such that 
   \[
      |\nabla_{c,\gamma}U^{c,\gamma}_{\theta}(y)|\le C, \quad \forall -1<y<1.
   \]
   Using this and  the fact that $U^{0,0}_{\theta}=0$, we have that for all $-1<y<1$, 
   \[
       |U^{c,\gamma}_{\theta}(y)|=|U^{c,\gamma}_{\theta}(y)-U^{0,0}_{\theta}(y)|\le \sup_{(c,\gamma)\in K, -1<y<1}|\nabla_{c,\gamma}U^{c,\gamma}_{\theta}(y)\cdot (c,\gamma)|\le C(|c|+|\gamma|).
   \]
   Thus 
   \begin{equation}\label{eqS_0_1}
      \|U^{c,\gamma}_{\theta}\|_{L^{\infty}(-1,1)}\le C(|c|+|\gamma|).
   \end{equation}
   For simplicity we use $U_{\theta}$ to denote $U^{c,\gamma}_{\theta}$. By (\ref{eqS_1}), we have 
   \[
      U'_{\theta}+\frac{U_{\theta}-4}{2(1-y^2)}U_{\theta}=c_3-\frac{2U_{\theta}}{1-y}.
   \]
   Let 
   \begin{equation}\label{eqS_0_2}
      w(y):=\int_{0}^{y}\frac{U_{\theta}-4}{2(1-s^2)}ds=\ln\frac{1-y}{1+y}+\int_{0}^{y}\frac{U_{\theta}(s)}{2(1-s^2)}ds.
   \end{equation}
   Since $U_{\theta}(0)=\gamma$, we have 
   \begin{equation}\label{eqS_0_3}
      U_{\theta}(y)=\gamma e^{-w}+e^{-w}\int_{0}^{y}e^{w}(c_3-\frac{2U_{\theta}}{1-s})ds.
   \end{equation}
   By (\ref{eqS_0_1}), for any fixed $0<\epsilon<1/2$, %there exists some $r(\epsilon)$, such that 
   \[
       \|U_{\theta}\|_{L^{\infty}(-1,1)}<4\epsilon, \quad \forall |(c,\gamma)|<\epsilon/C.
   \]
   By the above and (\ref{eqS_0_2}), we have that for $-1<y<0$, 
   \[
      |w(y)+\ln\frac{1+y}{1-y}|\le 2\epsilon\int_{y}^{0}\frac{ds}{1-s^2}=\epsilon\ln\frac{1-y}{1+y}.
   \]
    So 
   \[
      e^w\le 4(1+y)^{-1-\epsilon}, \quad e^{-w}\le (1+y)^{1-\epsilon}, \quad -1<y\le 0.
   \]
   By (\ref{eqS_0_1}), (\ref{eqS_0_3}) and the above, we have 
    \begin{equation}\label{eqS_0_4}
         |U_{\theta}|  \le C\gamma(1+y)^{1-\epsilon}+C(|c|+|\gamma|)e^{-w(y)}\int_{y}^{0} e^{w(s)}ds\\
                            %& C\gamma(1+y)^{1-\epsilon}+C(|c|+|\gamma|)(1+y)^{1-\epsilon}\int_{0}^{y} (1+s)^{-1-\epsilon}ds\\
                             \le \frac{C}{\epsilon}(|c|+|\gamma|)(1+y)^{1-2\epsilon}. 
   \end{equation}
%   \begin{equation}\label{eqS_0_4}
%       \begin{split}
%         |U_{\theta}| & \le C\gamma(1+y)^{1-\epsilon}+C(|c|+|\gamma|)e^{-w(y)}\int_{y}^{0} e^{w(s)}ds\\
%                            & C\gamma(1+y)^{1-\epsilon}+C(|c|+|\gamma|)(1+y)^{1-\epsilon}\int_{0}^{y} (1+s)^{-1-\epsilon}ds\\
%                            & \le \frac{C}{\epsilon}(|c|+|\gamma|)(1+y)^{1-2\epsilon}
%      \end{split}
%   \end{equation}
Denote $\mu_1=\int_{0}^{-1}\frac{U_{\theta}(s)}{2(1-s^2)}ds$. By (\ref{eqS_0_4}) we have $\mu_1=O(1)(|c|+|\gamma|)/\epsilon$. %\frac{|c|+|\gamma|}{\epsilon}$. 
  Then by (\ref{eqS_0_2}) and (\ref{eqS_0_4}), we have
\[
  %\begin{split}
   w(y)  =\ln\frac{1-y}{1+y}+\mu_1+\int_{-1}^{y}\frac{U_{\theta}(s)}{2(1-s^2)}ds
  % \ln\frac{1-y}{1+y}+\int_{0}^{-1}\frac{U_{\theta}(s)}{2(1-s^2)}ds+\int_{-1}^{y}\frac{U_{\theta}(s)}{2(1-s^2)}ds\\
            =\ln\frac{1-y}{1+y}+\mu_1+O(1)\frac{|c|+|\gamma|}{\epsilon}(1+y)^{1-2\epsilon}.
  % \end{split}
\]
Then we have 
\[
  \begin{split}
  &  e^{w}=\frac{1-y}{1+y}e^{\mu_1}(1+O(1)\frac{|c|+|\gamma|}{\epsilon}(1+y)^{1-2\epsilon}),\\ %=\frac{2}{1+y}e^{\mu_1}(1+O(1)\frac{|c|+|\gamma|}{\epsilon}(1+y)^{1-2\epsilon}),\\
  & e^{-w}=\frac{1+y}{1-y}e^{-\mu_1}(1+O(1)\frac{|c|+|\gamma|}{\epsilon}(1+y)^{1-2\epsilon}). %=\frac{1+y}{2}e^{-\mu_1}(1+O(1)\frac{|c|+|\gamma|}{\epsilon}(1+y)^{1-2\epsilon}).
  \end{split}
\]
Using the above, (\ref{eqS_0_3}) and  (\ref{eqS_0_4}),  we have that for $-1<y\le 0$, 
\begin{equation*}%\label{eqS_0_5}
   \begin{split}
       U_{\theta}(y) & =\gamma e^{-w(y)}+c_3e^{-w(y)}\int_{0}^{y}e^{w(s)}ds-e^{-w(y)}\int_{0}^{y}e^{w(s)}\frac{2U_{\theta}}{1-s}ds\\
%                            & =O(1)\gamma (1+y)+c_3\frac{1+y}{2}e^{-\mu_1}\int_{0}^{y}\frac{1-s}{1+s}e^{\mu_1}ds+O(1)\frac{|c|+|\gamma|}{\epsilon}(1+y)\\
%                            &+c_3\frac{1+y}{2}e^{-\mu_1}(1+O(1)\frac{|c|+|\gamma|}{\epsilon}(1+y)^{1-2\epsilon}) \int_{0}^{y}\frac{2}{1+s}e^{\mu_1}(1+O(1)\frac{|c|+|\gamma|}{\epsilon}(1+s)^{1-2\epsilon})ds\\
%                            & +\frac{1+y}{2}e^{-\mu_1}(1+O(1)\frac{|c|+|\gamma|}{\epsilon}(1+y)^{1-2\epsilon})\int_{0}^{y}\frac{2}{1+s}e^{\mu_1}(1+O(1)\frac{|c|+|\gamma|}{\epsilon}(1+s)^{1-2\epsilon})(1+s)^{1-2\epsilon}ds\\
%                            & =O(1)\frac{|c|+|\gamma|}{\epsilon} (1+y)+c_3(1+y)(1+O(1)\frac{|c|+|\gamma|}{\epsilon}(1+y)^{1-2\epsilon})(\ln(1+y)+O(1))\\
                            & =c_3(1+y)\ln(1+y)+O(1)\frac{|c|+|\gamma|}{\epsilon} (1+y).
   \end{split}
\end{equation*}
Estimate (\ref{eqS_0_0_1}) is established. 
%\begin{equation}\label{eqS_0_5}
%   \begin{split}
%       U_{\theta}(y) & =\gamma e^{-w(y)}+c_3e^{-w(y)}\int_{0}^{y}e^{w(s)}ds-e^{-w(y)}\int_{0}^{y}e^{w(s)}\frac{2U_{\theta}}{1-s}ds\\
%                            & =O(1)\gamma (1+y)\\
%                            &+c_3\frac{1+y}{2}e^{-\mu_1}(1+O(1)\frac{|c|+|\gamma|}{\epsilon}(1+y)^{1-2\epsilon}) \int_{0}^{y}\frac{2}{1+s}e^{\mu_1}(1+O(1)\frac{|c|+|\gamma|}{\epsilon}(1+s)^{1-2\epsilon})ds\\
%                            & +\frac{1+y}{2}e^{-\mu_1}(1+O(1)\frac{|c|+|\gamma|}{\epsilon}(1+y)^{1-2\epsilon})\int_{0}^{y}\frac{2}{1+s}e^{\mu_1}(1+O(1)\frac{|c|+|\gamma|}{\epsilon}(1+s)^{1-2\epsilon})(1+s)^{1-2\epsilon}ds\\
%                            & =O(1)\frac{|c|+|\gamma|}{\epsilon} (1+y)+c_3(1+y)(1+O(1)\frac{|c|+|\gamma|}{\epsilon}(1+y)^{1-2\epsilon})(\ln(1+y)+O(1))\\
%                            & =c_3(1+y)\ln(1+y)+O(1)\frac{|c|+|\gamma|}{\epsilon} (1+y).
%   \end{split}
%\end{equation}
%Similarly, we have that for $0\le x<1$, 
%\begin{equation}\label{eqS_0_5_1}
%   U_{\theta}(y)=c_3(1-y)\ln(1-y)+O(1)\frac{|c|+|\gamma|}{\epsilon} (1-y).
%\end{equation}   
%   By  (\ref{eqS_0_5}) and (\ref{eqS_0_5_1}), we have (\ref{eqS_0_0_1}).

%  By (\ref{eqS_0_5}) and (\ref{eqS_0_5_1}), we have
%   \begin{equation}\label{eqS_0_5_4}
%     \|(1-x^2)^{-1}(|\ln(1+x)|+|\ln(1-x)|+1)^{-1}U_{\theta}\|_{L^{\infty}(-1,1)}\le C(|c|+|\gamma|).
%   \end{equation}
   
   Next, we make the estimate of $U'_{\theta}$ and prove (\ref{eqS_0_0_2}). By (\ref{eqS_1}) and (\ref{eqS_0_0_1}), we have that for $-1<y\le 0$, 
   \begin{equation*}%\label{eqS_0_6}
     % \begin{split}
         U_{\theta}'  =c_3-\frac{1}{1-y^2}(\frac{1}{2}U_{\theta}^2+2yU_{\theta})
%                           & =c_3-\frac{1}{1-y^2}(\frac{1}{2}U_{\theta}^2-2U_{\theta}+2(1+y)U_{\theta})\\
%                          & =c_3-\frac{1}{1-y^2}(\frac{1}{2}c_3(1+y)^2(\ln(1+y))^2+O(1)\frac{|c|+|\gamma|}{\epsilon} (1+y)^2\ln(1+x)-2c_3(1+y)\ln(1+y)\\
%                          & +O(1)\frac{|c|+|\gamma|}{\epsilon} (1+y))\\
                           =c_3\ln(1+y)+O(1)(|c|+|\gamma|). 
                         %  \end{split}
   \end{equation*}
   Estimate (\ref{eqS_0_0_2}) is established. 
   
   Differentiating (\ref{eqS_1}), and using (\ref{eqS_0_0_1}) and (\ref{eqS_0_0_2}), we have for $-1<y\le 0$ that 
   \[
     %\begin{split}
      (1-y^2)U_{\theta}''  =-2c_3y-U_{\theta}U'_{\theta}-2U_{\theta}
                                     =2c_3+O(1)(|c|+|\gamma|)(1+y)|\ln(1+y)|^2.
      %\end{split}
   \]
   Estimate (\ref{eqS_0_0_3}) follows immediately.  The lemma is proved. 

\end{proof}

   \begin{cor}\label{corS_2}
  Let $K$ be a compact subset of $M$. Then there exist some positive constant $C$, depending only on $K$, such that for all $(c,\gamma)$ in $K$, and $x$ in $\mathbb{R}^3\setminus\{x'=0\}$.  
  \begin{equation}\label{eqS_2_0_1}
     u^{c,\gamma}_{\theta}(x)=-\frac{c_3\mathrm{sgn}(x_3)|x'|}{|x|^2}\ln \frac{|x'|}{|x|}+\frac{O(1)(|c|+|\gamma|)|x'|}{|x|^2},%\frac{c_3|x'|}{|x|^2}\ln \frac{|x'|}{|x|}+\frac{O(1)(|c|+|\gamma|)|x'|}{|x|^2},
  \end{equation}
  \begin{equation}\label{eqS_2_0_2}
     u^{c,\gamma}_{r}(x)=\frac{2c_3}{|x|}\ln\frac{|x'|}{|x|}+\frac{O(1)(|c|+|\gamma|)}{|x|}, %\frac{2c_3\mathrm{sgn}(x_3)}{|x|}\ln\frac{|x'|}{|x|}+\frac{O(1)(|c|+|\gamma|)}{|x|}
%     \left\{
%       \begin{split}
%          & -\frac{2c_3}{|x|}\ln\frac{|x'|}{|x|}+\frac{O(1)(|c|+|\gamma|}{|x|}, \quad x_3\ge 0, \\
%          & \frac{2c_3}{|x|}\ln\frac{|x'|}{|x|}+\frac{O(1)(|c|+|\gamma|}{|x|}, \quad x_3<0,
%       \end{split}
    % \right\}
  \end{equation}
  and 
  \begin{equation}\label{eqS_2_0_3}
     |\nabla u (x)|=\frac{2|c_3|}{|x||x'|}+\frac{O(1)(|c|+|\gamma|)}{|x|^2}\ln\frac{|x|}{|x'|}. %\frac{O(1)c_3}{|x|^2}\ln\frac{|x'|}{|x|}+\frac{O(1)(|c|+|\gamma|)}{|x|^2}.
  \end{equation}
%   \[
%     |u^{c,\gamma}|\le \frac{C(|c|+|\gamma|)(|\ln(|x'|/|x|)|+1)}{|x|}, \textrm{ and }|\nabla u^{c,\gamma}|\le \frac{C(|c|+|\gamma|)}{|x||x'|}, x\in \mathbb{R}^3\setminus\{x_3=0\}.
%  \]
  %Let $u^{c,\gamma}$ be a solution of (\ref{NS}) given by the $U^{c,\gamma}_{\theta}$ of (\ref{eqS_1}), then there exists some constant $K(c,\gamma)$, such that 
%  \[
%     |u^{c,\gamma}|\le \frac{K(c,\gamma)(|\ln(|x'|/|x|)|+1)}{|x|}, \textrm{ and }|\nabla u^{c,\gamma}|\le \frac{K(c,\gamma)}{|x||x'|}, x\in \mathbb{R}^3\setminus\{x_3=0\}.
%  \]
%  Moreover, $\lim_{(c,\gamma)\to 0}K(c,\gamma)=0$.
\end{cor}
\begin{proof}
   For convenience write $u^{c,\gamma}=u$. By definition, $u=u_re_r+u_{\theta}e_{\theta}$, where $u_r=\frac{1}{r}U'_{\theta}$, $u_{\theta}=\frac{1}{r\sin\theta}U_{\theta}$. Denote $y=\cos\theta$, by Lemma \ref{lemS_0}, we have 
   \[
      U^{c,\gamma}_{\theta}(y)=-c_3\mathrm{sgn}(\cos\theta)\sin^2\theta \ln\sin\theta+O(1)(|c|+|\gamma|) \sin^2\theta.
   \]
   Since $r=|x|$ and $|x'|=|x|\sin\theta$, estimate (\ref{eqS_2_0_1}) follows from the above. Estimate (\ref{eqS_2_0_2}) follows from (\ref{eqS_0_0_2}).

   Next, we compute
   \begin{equation*}%\label{eqS_2_4}
      \nabla u=\nabla u_r e_r+u_r\nabla e_r+\nabla u_{\theta} e_{\theta}+u_{\theta}\nabla e_{\theta}.
   \end{equation*}
   By (\ref{eqS_0_0_1}) and (\ref{eqS_0_0_2}), we have 
   \begin{equation*}%\label{eqS_2_5}
     \begin{split}
      |\nabla u_r| & =|\frac{\partial u_r}{\partial r}e_r+\frac{1}{r}\frac{\partial u_{r}}{\partial \theta}e_{\theta}|
                  =|-\frac{1}{r^2}U'_{\theta}(y)e_{r}+\frac{1}{r^2}U''_{\theta}(y)(-\sin\theta)e_{\theta}|\\
                        & =\frac{2|c_3|}{r^2\sin\theta}+O(1)\frac{(|c|+|\gamma|)}{r^2}\ln\sin\theta
                         =\frac{2|c_3|}{|x||x'|}+O(1)\frac{|c|+|\gamma|}{|x|^2}\ln\frac{|x|}{|x'|},
%                 & =-\frac{1}{r^2}(2c_3\mathrm{sgn}(\theta-\frac{\pi}{2})\ln\sin\theta+O(1)(|c|+|\gamma|))e_{r}+\frac{1}{r^2}(\frac{2c_3}{\sin^2\theta}+O(1)(|c|+|\gamma|)|\ln\sin\theta|^2)(-\sin\theta)e_{\theta} \\
%                 & = -\frac{2c_3}{r^2\sin\theta}e_{\theta}-\frac{2c_3\mathrm{sgn}(\theta-\frac{\pi}{2})\ln\sin\theta}{r^2}+\frac{O(1)(|c|+|\gamma|)}{r^2}
                  \end{split}
   \end{equation*}   
   and 
  % By (\ref{eqS_2_1}) and (\ref{eqS_2_2}), we have 
   \begin{equation*}%\label{eqS_2_6}
     \begin{split}
        |\nabla u_{\theta}| & =|\frac{\partial u_{\theta}}{\partial r}e_r+\frac{1}{r}\frac{\partial u_{\theta}}{\partial \theta}e_{\theta}|
                            =|-\frac{1}{r^2}U_{\theta}e_r+\frac{1}{r^2}U'_{\theta}(-\sin\theta)e_{\theta}|\\
                           & \le \frac{C(|c|+|\gamma|)}{|x|^2}\frac{|x'|}{|x|}\ln\frac{|x|}{|x'|}. 
%                           & =-\frac{1}{r^2}(c_3\sin^2\theta\ln\sin\theta+O(1)(|c|+|\gamma|) \sin^2\theta)e_r\\
%                           & +\frac{1}{r^2}(2c_3\mathrm{sgn}(\theta-\frac{\pi}{2})\ln\sin\theta
%                            +O(1)(|c|+|\gamma|))(-\sin\theta)e_{\theta}\\
%                           & =-\frac{2c_3\mathrm{sgn}(\theta-\frac{\pi}{2})}{r^2}\sin\theta\ln\sin\theta -\frac{c_3}{r^2}\sin^2\theta\ln\sin\theta
%                            +\frac{O(1)(|c|+|\gamma|)\sin^2\theta}{r^2}.
     \end{split}
   \end{equation*}
   Since $|\nabla e_r|+|\nabla e_{\theta}|\le C/r$, estimate (\ref{eqS_2_0_3}) follows from the above. %$|\nabla e_r|+|\nabla e_{\theta}|= \frac{O(1)}{r}$. So by (\ref{eqS_2_4})-(\ref{eqS_2_6}), (\ref{eqS_2_0_1}) and (\ref{eqS_2_0_2}), we have 
%   \[
%     \begin{split}
%        |\nabla u| & =\frac{O(1)c_3}{r^2}\sin\theta\ln\sin\theta+O(1)(|c|+|\gamma|) \frac{\sin\theta}{r^2}+\frac{2c_3O(1)}{r^2}\ln\sin\theta+\frac{O(1)(|c|+|\gamma|)}{r^2}\\
%                        & +\frac{2|c_3|}{r^2\sin\theta}+\frac{2|c_3\ln\sin\theta|}{r^2}+\frac{O(1)(|c|+|\gamma|)}{r^2}\\
%                        & +\frac{2|c_3|}{r^2}|\sin\theta\ln\sin\theta|+\frac{c_3}{r^2}|\sin^2\theta\ln\sin\theta|
%                            +\frac{O(1)(|c|+|\gamma|)\sin^2\theta}{r^2}\\
%                            & =\frac{2|c_3|}{r^2\sin\theta}+\frac{O(1)c_3}{r^2}\ln\sin\theta+\frac{O(1)(|c|+|\gamma|)}{r^2}\\ 
%                            & =\frac{2|c_3|}{|x||x'|}+\frac{O(1)c_3}{|x|^2}\ln\frac{|x'|}{|x|}+\frac{O(1)(|c|+|\gamma|)}{|x|^2}.
%     \end{split}
%   \]
%   (\ref{eqS_2_0_3}) is proved. The corollary is proved.
   \end{proof}

\section{Force of $u^{c,\gamma}$, $(c,\gamma)\in M$}\label{sec_F}
In this section, we study the force of the special solutions $u^{c,\gamma}$ and prove Proposition \ref{prop_force}, where $(c,\gamma)$ in $M$ and $M$ is the set defined by (\ref{eqS_M}).  
%In this section, we study the force of the special solutions $u^{c,\gamma}$ we have introduced in Section \ref{sec_1}, where $(c,\gamma)$ in $M$ and $M$ is the set defined by (\ref{eqS_M}). % where $c=(0,0,c_3)$, $c_3>-4$, and $\gamma^-(c)<\gamma<\gamma^+(c)$.  
 Recall that
 $(r,\theta, \phi)$ are the polar coordinates, let $\rho=r\sin\theta$, $(\rho, \phi, z)$ be the cylindrical coordinates, $y=\cos\theta$. Recall $(u^{c,\gamma}, p^{c,\gamma})$ are given by (\ref{eq_u_cgamma}), %$u^{c,\gamma}$ can be expressed as $u^{c,\gamma}=\frac{1}{r}(U^{c,\gamma}_{\theta})'(y)e_r+\frac{1}{\sin\theta}U^{c,\gamma}_{\theta}(y)e_{\theta}$, 
 where $U^{c,\gamma}_{\theta}(y)$ is a solution of (\ref{eqS_1}). For convenience, denote 
 $u=u^{c,\gamma}$, $p=p^{c,\gamma}$ and $U_{\theta}=U_{\theta}^{c,\gamma}$.  In Euclidian coordinates, $x=(x_1,x_2, x_3)$ and $u=(u_1, u_2, u_3)$. %, write $\partial_{x_i}=\partial_{x_i}=\frac{\partial}{\partial x_i}$, $i=1,2,3$. 

 \noindent\emph{Proof of Proposition \ref{prop_force}}:
 Let $(c,\gamma)\in M$. For any $R>0$, let 
 \begin{equation}\label{eq_O}
   \Omega:=\{x\in \mathbb{R}^3| |x'|\le R, -R<x_3<R\}. 
\end{equation}
We prove  (\ref{eqF_2}) and (\ref{eqF_3}) for any $\varphi\in C_c^{\infty}(\Omega)$. 
 Throughout the proof we denote $O(1)$ as some quantity satisfying $|O(1)|\le C$ for some $C>0$ depending only on $(c,\gamma)$, $R$ and $\varphi$. 
%To prove Proposition \ref{prop_force}, we prove (\ref{eqF_2}) and (\ref{eqF_3}) for any $\varphi\in C_c^{\infty}(\mathbb{R}^3)$. Let $K$ be a compact set in $M$. Let $O(1)$ denote some quantity  satisfying $|O(1)|\le C(|c|+|\gamma|)$ for some positive constant $C$ depending only on $K$. Let $\varphi\in C_c^{\infty}(\mathbb{R}^3)$, then there exists some $R>0$, such that $\mathrm{supp} \varphi\subset \Omega$, where
%\[
%   \Omega:=\{x\in \mathbb{R}^3| |x'|\le R, -R<x_3<R\}. 
%\]
%For any $\epsilon>0$, denote
%\[
%   \Omega_{\epsilon}:=\{x\in \mathbb{R}^3| |x'|\le \epsilon, -R<x_3<R\}.
%\]
%Let $o_{\epsilon}(1)$ be a function where $o_{\epsilon}(1)\to 0$ as $\epsilon\to 0$ uniform for $(x,\gamma)$ in $K$.
%Let $o_{\epsilon}(1)$ be a function where $o_{\epsilon}(1)\to 0$ as $\epsilon\to 0$ uniform for $(x,\gamma)$ in $K$. %, and $C$ be a constant independent of $\epsilon$.

By Lemma \ref{lemS_0}, %we have that for any $(c,\gamma)$ in $M$, 
\begin{equation}\label{eqA_1}
   |U_{\theta}(y)| 
   =O(1)\sin^2\theta|\ln\sin\theta|, \textrm{ } |(U_{\theta})'(y)|=O(1)|\ln\sin\theta|,  \textrm{ }  |(U_{\theta})''(y)|= \frac{O(1)}{\sin^2\theta}.
\end{equation}
Recall that here '' $'$ " denote the derivative with respect to $y$.  
By Corollary \ref{corS_2} and (\ref{eq_u_cgamma}), we have
\begin{equation}\label{eqA_u}
    |u_{\theta}|=\frac{O(1)\sin\theta |\ln \sin\theta|}{r}, \textrm{ }   |u_{r}|=\frac{O(1)|\ln \sin\theta|}{r},  \textrm{ }  |\nabla u|=\frac{O(1)}{r^2\sin\theta},  \textrm{ }   |p|=\frac{O(1)|\ln\sin\theta|}{r^2}. 
  \end{equation}
%\begin{equation}\label{eqA_u}
%  \begin{split}
%   & |u_{\theta}|=\frac{|c_3|}{r}|\sin\theta\ln \sin\theta|+\frac{O(1)\sin\theta}{r},\\
%   & |u_{r}|=\frac{2|c_3|}{r}|\ln\sin\theta|+\frac{O(1)}{r},\\
%   &  |\nabla u|=\frac{2|c_3|}{r^2\sin\theta}%+\frac{O(1)c_3}{|x|^2}\ln\frac{|x'|}{|x|}
%     +\frac{O(1)|\ln\sin\theta|}{r^2}.
%   \end{split}
%\end{equation}

%So by (\ref{eq_u_cgamma}) and (\ref{eqA_u}), we have 
%\begin{equation}\label{eqA_p}
%   p=\frac{O(1)\ln|\sin\theta|}{r^2}%\frac{2|c_3\ln\sin\theta|}{r^2}+\frac{O(1)}{r^2}. %O(1)\frac{|\ln\sin\theta|}{r^2}.
%\end{equation}

We first prove (\ref{eqF_3}).  
For any $\epsilon>0$, denote 
\[
   \Omega_{\epsilon}:=\{x\in \mathbb{R}^3| |x'|\le \epsilon, -R<x_3<R\}.
\]
Let $o_{\epsilon}(1)$ be a function where $o_{\epsilon}(1)\to 0$ as $\epsilon\to 0$.
%
%For any  $\varphi\in C^{\infty}_c(\mathbb{R}^3)$, let $\mathrm{supp} \varphi\subset \Omega$, then  since 
Since $u\in C^{\infty}(\mathbb{R}^3\setminus\{x'=0\})$, we have $\dive{}\textrm{ }u=0$ in $\mathbb{R}^3\setminus\{x'=0\}$. Therefore
\[
 % \begin{split}
   \int_{\mathbb{R}^3}u\cdot \nabla \varphi dx  =\int_{\Omega\setminus\Omega_{\epsilon}}u\cdot \nabla \varphi dx+\int_{\Omega_{\epsilon}}u\cdot \nabla \varphi dx =-\int_{\partial \Omega_{\epsilon}\cap\{|x'|=\epsilon\}}u\cdot \nabla \varphi dx+\int_{\Omega_{\epsilon}}u\cdot \nabla \varphi dx. 
                                                                       % & =-\int_{\partial \Omega_{\epsilon}\cap\{|x'|=\epsilon\}}u\cdot \nabla \varphi dx+\int_{\Omega_{\epsilon}}u\cdot \nabla \varphi dx. 
  % \end{split}
\]
By (\ref{eqA_u}), we have $|u|\le C/|x|$. So 
\[
  \int_{\partial \Omega_{\epsilon}}|u\cdot \nabla \varphi| dx\le \int_{\partial\Omega_{\epsilon}}\frac{C}{|x|}d\sigma(x)\le C\epsilon|\ln\epsilon|, \quad %C\epsilon\int_{-R}^{R}\frac{1}{\sqrt{\epsilon^2+x_3^2}}dx_3\le C\epsilon\ln\epsilon,
   \int_{\Omega_{\epsilon}}|u\cdot \nabla \varphi| dx \le \int_{\Omega_{\epsilon}}\frac{C}{|x|}dx\le C\epsilon^2|\ln\epsilon|. %\le C\int_{-R}^{R}\int_{0}^{\epsilon}\frac{1}{\sqrt{\rho^2+x_3^2}}\rho d\rho dx_3\le C\epsilon^2\ln\epsilon.
\]
Sending $\epsilon$ to $0$ in the above leads to (\ref{eqF_3}). 
%Using the above estimates and the fact that $\epsilon$ is an arbitrary number we have (\ref{eqF_3}). 
%\[
%   \int_{\mathbb{R}^3}u\cdot \nabla \varphi dx=0. 
%\]
%So (\ref{eqF_3}) is proved. 

Next, we prove (\ref{eqF_2}). Denote the stress tensor
\[
   T_{ij}(u):=p\delta_{ij}+u_iu_j-(\partial_{x_j}u_i+\partial_{x_i}u_j). %(\frac{\partial u_i}{\partial x_j}+\frac{\partial u_j}{\partial x_i}).
\]
Then (\ref{eqF_2}) is equivalent to 
\begin{equation}\label{eqF}
    \int_{\Omega}T_{ij}(u)\partial_{x_i}\varphi dx=[4\pi c_3\int_{-R}^{R}\ln |x_3|\partial_{x_3}\varphi(0,0,x_3)dx_3-b\varphi(0)]\delta_{j3}e_3, \quad \forall \varphi\in C_c^{\infty}(\Omega),
 \end{equation}
 where $b=b^{c,\gamma}$ is given by (\ref{eqA_b}). 
 
 %To prove (\ref{eqF}), we first prove the following lemmas.
 \noindent\textbf{Claim $1$}:  $T_{ij}(u)\in L^q_{loc}(\mathbb{R}^3)$, for any $q<\frac{3}{2}$.  %For any R and $q<\frac{3}{2}$, $T_{ij}(u)\in L^q(\Omega)$.
 
To prove the Claim, notice that by (\ref{eqA_u}), we have that
\begin{equation}\label{eqA_T_esti}
   |T_{ij}|\le |p|+|u|^2+2|\nabla u|\le \frac{C}{r^2\sin\theta}. 
\end{equation} 
So for any $R>0$ and $\Omega$ defined by (\ref{eq_O}), we have, using $q<\frac{3}{2}$,  
\[
   \int_{\Omega}|T_{ij}|^q  \le C\int_{B_{2R}}\frac{1}{r^{2q}|\sin\theta|^q}dx 
     =C\int_{0}^{R}\int_{0}^{\pi}\int_{0}^{2\pi}\frac{1}{r^{2q-2}|\sin\theta|^{q-1}}r^2\sin\theta d\phi d\theta dr \le C.
\]
The Claim is proved.

Using  Claim $1$ and the fact that $\partial_{x_i}T_{ij}=0$ in $\mathbb{R}^3\setminus\{x'=0\}$ for any $1\le j\le 3$,  we have that
\begin{equation*}%\label{eqA_5}
% \begin{split}
   -\int_{\Omega}T_{ij}\partial_{x_i}\varphi dx  =-\int_{\Omega\setminus \Omega_{\epsilon}}T_{ij}\partial_{x_i}\varphi dx-\int_{\Omega_\epsilon}T_{ij}\partial_{x_i}\varphi dx
    %& =\int_{\partial \Omega_{\epsilon}}T_{\ij}\cdot \nu_i \varphi dx+\int_{\Omega\setminus \Omega_{\epsilon}}\partial_{x_i} T_{ij}\varphi dx-\int_{\Omega_\epsilon}T_{ij}\partial_{x_i}\varphi dx\\
      =\int_{\partial \Omega_{\epsilon}}T_{\ij}\cdot \nu_i \varphi dx-\int_{\Omega_\epsilon}T_{ij}\partial_{x_i}\varphi dx. %\\
     %& =:L_j+o_{\epsilon}(1), %=:L+H
  % \end{split}
\end{equation*}
Let 
\[
   L_j : =\int_{\partial \Omega_{\epsilon}}T_{\ij}\cdot \nu_i \varphi dx. 
\]
Since  $T_{ij}\in L^1(\Omega)$, we have $\int_{\Omega_\epsilon}T_{ij}\partial_{x_i}\varphi dx=o_{\epsilon}(1)$. So for each $j=1, 2, 3$, 
\begin{equation}\label{eqA_5}
   -\int_{\Omega}T_{ij}\partial_{x_i}\varphi dx =L_j+o_{\epsilon}(1). 
\end{equation}
%we have that $H\to 0$ as $\epsilon\to 0$.
%Let
%\[
%   L_j:=\int_{\partial \Omega_{\epsilon}}T_{\ij}\cdot \nu_i \varphi d\sigma.
%\]
%Then
By computation
\[
  %\begin{split}
   L_j %& =\int_{\partial \Omega_{\epsilon}}T_{\ij}\cdot \nu_i (\varphi(0,0,x_3)+\nabla' \varphi (\xi', x_3)\cdot x') d\sigma\\
       =\int_{\partial \Omega_{\epsilon}\cap\{|x'|=\epsilon\}}T_{\ij}\cdot \nu_i \varphi(0,0,x_3)+ O(1) \epsilon\int_{\partial \Omega_{\epsilon}\cap\{|x'|=\epsilon\}}|T_{\ij}|
       = :L_j^{(1)}+L_j^{(2)}. 
  % \end{split}
\]
By (\ref{eqA_T_esti}), we have that for $j=1,2,3$, 
\begin{equation}\label{eqA_4}
    |L_j^{(2)}|\le C\epsilon \int_{\partial \Omega_{\epsilon}\cap\{|x'|=\epsilon\}}|T_{\ij}|d\sigma\le C\int_{-R}^{R}\frac{\epsilon}{\sqrt{\epsilon^2+x_3^2}}dx_3\le C\epsilon|\ln\epsilon|\to 0, \textrm{ as } \epsilon\to 0. 
\end{equation}
%So 
%\begin{equation}\label{eqA_4}
% \lim_{\epsilon\to 0}L_j^{(2)}=0, \quad j=1,2,3.
%\end{equation}

%Next, we prove (\ref{eqF_2}) by the following lemmas.

\begin{lem}\label{lemA_5}
   \[
      L^{(1)}_j=0, \quad j=1,2. %\lim_{\epsilon\to 0} L^{(1)}_j=0, \quad j=1,2
      \]
\end{lem}
\begin{proof}
   %By Lemma \ref{lemA_3} we have $L_j^{(2)}\to 0$ as $\epsilon \to 0$, $j=1,2$. So we only need to prove $L_j^{(1)}\to 0$ as $\epsilon\to 0$, $j=1,2$. 
   We will show that $T_{ij}\cdot\nu_i=F(|x'|, x_3)x_{j}$ for some function $F(|x'|, x_3)$, $j=1,2$, so its integral on any cylinder $\{|x'|=\epsilon\}$ vanishes.  
   Let $x'=(x_1,x_2)$, $u'=(u_1,u_2)$, $\nabla'=(\partial_1, \partial_2)$, $(\rho, \phi, z)$ be the cylindrical coordinates, and the unit normal
$
   e_{\rho}=(\cos\phi,\sin\phi,0), e_{\phi}=(-\sin\phi, \cos\phi, 0),  e_{z}=(0,0,1). 
$
So we have $x=\rho e_{\rho}+ze_z, \textrm{  }x'=\rho e_{\rho}$. Notice $u$ is axisymmetric no-swirl,  %$u_{\phi}=0$, 
 we can write
$
   u=u_{\rho}e_{\rho}+u_ze_z, 
$
where $u_{\rho}$ and $u_{z}$ are both independent of $\phi$. 
%and
%\[
%   x=\rho e_{\rho}+ze_z, \textrm{  }x'=\rho e_{\rho}. 
%\]
By computation, 
\[
    x'\cdot u'=\rho u_{\rho},\quad x'\cdot\nabla' u=\rho\frac{\partial u_{\rho}}{\partial \rho} e_{\rho}+\rho\frac{\partial u_{z}}{\partial \rho} e_{z}, 
\]
%
%\[
%    x'\cdot\nabla' u=\rho\frac{\partial u_{\rho}}{\partial \rho} e_{\rho}+\rho\frac{\partial u_{z}}{\partial \rho} e_{z}
%\]
and
\[
   \nabla (x'\cdot u')=\nabla(\rho u_{\rho})=\frac{\partial (\rho u_{\rho})}{\partial \rho}e_{\rho}+\frac{\partial (\rho u_{\rho})}{\partial z}e_z.
\]
%Here $\nabla'=(\partial_1, \partial_2)$. 
On $\partial \Omega_{\epsilon}\cap \{|x'|=\epsilon\}$, the outer-normal $\nu=\frac{1}{\rho}(x_1,x_2, 0)$. Since $u$ is axisymmetric, $u_{\rho}$ is independent of $\phi$,  so $u_1=u_{\rho}(\rho, z)\cos\phi$, and we have 
\[
  \begin{split}
   %\sum_{i=1}^{2}
   T_{i1}\cdot \nu_i & =\frac{1}{\rho}\left(px_1+x'\cdot u'u_1-x'\cdot\nabla' u_1-\partial_1(x'\cdot u')+u_1\right)\\
                    & =\frac{1}{\rho}\left(p\rho\cos\phi+\rho u_{\rho}u_\rho\cos\phi-\rho\frac{\partial u_{\rho}}{\partial \rho}\cos\phi-\frac{\partial (\rho u_{\rho})}{\partial \rho}\cos\phi+u_{\rho}\cos\phi\right)\\
                   &=G(\rho,z)\cos\phi, 
  \end{split}
\]
where 
\[
   G(\rho, z)=\frac{1}{\rho}\left(p\rho+\rho u_{\rho}u_\rho-\rho\frac{\partial u_{\rho}}{\partial \rho}-\frac{\partial (\rho u_{\rho})}{\partial \rho}+u_{\rho}\right).
\]
%for some function $G(\rho,z)$.\\
So
\[
 % \begin{split}
    L_1^{(1)}= \int_{\rho=\epsilon}T_{i1}\varphi_1(0,0,z)\nu_i d\sigma
     =\epsilon\int_{-R}^{R}G(\epsilon,z)\varphi_1(0,0,z)dz\int_{0}^{2\pi}\cos\phi d\phi
     =0
 % \end{split}
\]
With similar argument we also have $L_2^{(1)}=0$. The lemma is proved.
\end{proof}

\begin{lem}\label{lemA_6}
%There exists some constant $b$, depending only on $u^{c,\gamma}$, such that
\[
  \lim_{\epsilon \to 0}L^{(1)}_3=4\pi c_3\int_{-R}^{R}\ln |x_3|\partial_{x_3}\varphi(0,0,x_3)dx_3-b\varphi(0), 
\]
where $b$ is the constant defined by (\ref{eqA_b}).
%\[
%   b=\int_{-1}^{1}\left(y|(U^{c,\gamma}_{\theta})'|^2-2U^{c,\gamma}_{\theta}-\frac{y}{2(1-y^2)}|U^{c,\gamma}_{\theta}|^2 \right)dy.
%\]
\end{lem}
\begin{proof}
 Recall 
 \[
    L_3^{(1)}=\frac{1}{\epsilon}\int_{\rho=\epsilon}(T_{13}x_1+T_{23}x_2)\varphi(0,0,x_3),
 \]
 and for $i=1,2$, 
 \[
    T_{i3}=u_iu_3-\frac{\partial u_i}{\partial x_3}-\frac{\partial u_3}{\partial x_i}.
 \]
 Since $u=(u_1, u_2, u_3)=\frac{1}{r}U_{\theta}'e_r+\frac{1}{r\sin\theta}U_{\theta}e_{\theta}$,  we have
\begin{equation*}%\label{eqA_4_1}
   \begin{split}
      & u_1(x_1,x_2,x_3)=\frac{x_1}{r^2}U'_{\theta}(y)+\frac{x_1x_3}{r\rho^2}U_{\theta}(y),\\
      & u_2(x_1,x_2,x_3)=\frac{x_2}{r^2}U'_{\theta}(y)+\frac{x_2x_3}{r\rho^2}U_{\theta}(y),\\
      & u_3(x_1,x_2,x_3)=\frac{x_3}{r^2}U'_{\theta}(y)-\frac{1}{r}U_{\theta}(y).
   \end{split}
\end{equation*}
Recall that $r^2=x_1^2+x_2^2+x_3^2$, $\rho^2=x_1^2+x_2^2$, $y=\cos\theta=\frac{x_3}{r}=\frac{x_3}{\sqrt{x_1^2+x_2^2+x_3^2}}$. By computation we have 
\[
   \frac{\partial u_i}{\partial x_3}=\frac{x_i}{r^3}U_{\theta}-\frac{x_ix_3}{r^4}U'_{\theta}+\frac{x_i\rho^2}{r^5}U''_{\theta}, \quad \frac{\partial u_3}{\partial x_i} =\frac{x_i}{r^3}U_{\theta}-\frac{x_ix_3}{r^4}U'_{\theta}-\frac{x_ix^2_3}{r^5}U''_{\theta}.
\]
%\begin{equation}\label{eqA_u_1}
%   \frac{\partial u_i}{\partial x_3}=-\frac{x_ix_3}{r^4}U'_{\theta}+\frac{x_i}{r^3}U_{\theta}+\frac{x_i\rho^2}{r^5}U''_{\theta},
%\end{equation}
%and
%\begin{equation}\label{eqA_u_2}
%   \frac{\partial u_3}{\partial x_i} =\frac{x_i}{r^3}U_{\theta}-\frac{x_ix^2_3}{r^5}U''_{\theta}.
%\end{equation}
So 
\begin{equation}\label{eqA_T_1}
      \sum_{i=1}^{2}T_{i3}x_i  =\frac{\rho^2x_3}{r^4}|U_{\theta}'|^2+\frac{x_3^2-\rho^2}{r^3}U_{\theta}U'_{\theta}-\frac{x_3}{r^2}U^2_{\theta}+\frac{2\rho^2x_3}{r^4}U'_{\theta}-\frac{2\rho^2}{r^3}U_{\theta}-\frac{\rho^2(\rho^2-x_3^2)}{r^5}U''_{\theta}.
      \end{equation}
Since $U_{\theta}$ satisfy (\ref{eqS_1}), take derivative of the first equation of (\ref{eqS_1}) both sides with respect to $y$, we have 
\[
   (1-y^2)U''_{\theta}+2U_{\theta}+U_{\theta}U'_{\theta}=-2c_3y.
\]
Plug in $U''_{\theta}=-\frac{1}{1-y^2}(2U_{\theta}+U_{\theta}U'_{\theta}+2c_3y)$ in (\ref{eqA_T_1}), we have 
\begin{equation*}%\label{eqA_T_2}
 % \begin{split}
   \sum_{i=1}^{2}T_{i3}x_i %& =\frac{1}{r}\left(-2c_3\cos\theta(1-2\sin^2\theta)  +\sin\theta^2\cos\theta(|U_{\theta}'|^2+U'_{\theta})-\cos\theta U_{\theta}^2-\cos^2\theta U_{\theta} \right)\\
                                          =-\frac{2c_3x_3}{r^2}+\left(\frac{\rho^2x_3}{r^4}(|U_{\theta}'|^2+2U'_{\theta}+4c_3)-\frac{x_3}{r^2}U^2_{\theta}-\frac{x_3^2}{r^3}U_{\theta} \right). 
                                          %=:-\frac{2c_3x_3}{r^2}+G(x).
 %  \end{split}
\end{equation*}
Let
\begin{equation}\label{eqA_G}
   G(x)=\frac{\rho^2x_3}{r^4}(|U_{\theta}'|^2+2U'_{\theta}+4c_3)-\frac{x_3}{r^2}U^2_{\theta}-\frac{x_3^2}{r^3}U_{\theta}. %\frac{1}{r}\left(\sin\theta^2\cos\theta(|U_{\theta}'|^2+U'_{\theta}+4c_3)-\cos\theta (U^2_{\theta}+\cos\theta U_{\theta}) \right).
\end{equation}
Now we have 
\[
   \begin{split}
      L_3^{(1)} & =\frac{1}{\epsilon}\int_{\rho=\epsilon}(T_{13}x_1+T_{23}x_2)\varphi(0,0,x_3)d\sigma\\
           % & =\frac{1}{\epsilon}\int_{\rho=\epsilon}(-\frac{2c_3x_3}{r^2}+R(x))\varphi(0,0,x_3)\\
            & =-\frac{1}{\epsilon}\int_{\rho=\epsilon}\frac{2c_3x_3}{r^2}\varphi(0,0,x_3)d\sigma+\frac{1}{\epsilon}\int_{\rho=\epsilon}G(x)\varphi(0,0,x_3)d\sigma\\
            & =:A+B.
               \end{split}
\]
Since $\varphi(0,0,R)=\varphi(0,0,-R)=0$, we have
\[
%  \begin{split}
   A  =-4\pi c_3\int_{-R}^{R}\frac{x_3}{\epsilon^2+x_3^2}\varphi(0,0,x_3)dx_3%-4\pi c_3(\int_{0}^{R}\frac{x_3}{\epsilon^2+x_3^2}\varphi(0,0,x_3)dx_3+\int_{-R}^{0}\frac{x_3}{\epsilon^2+x_3^2}\varphi(0,0,x_3)dx_3)\\
     %& =-2\pi c_3( \ln(\epsilon^2+R^2)\varphi(0,0,R)-\ln \epsilon^2 \varphi(0)-\int_{-R}^{R}\ln(\epsilon^2+x_3^2)\partial_{x_3}\varphi(0,0,x_3)dx_3\\
     %& +\ln \epsilon^2 \varphi(0)-\ln(\epsilon^2+R^2)\varphi(0,0,-R))\\
      =2\pi c_3\int_{-R}^{R}\ln(\epsilon^2+x_3^2)\partial_{x_3}\varphi(0,0,x_3)dx_3. 
   %\end{split}
\]

So 
\begin{equation}\label{eqA_A}
   \lim_{\epsilon\to 0} A=4\pi c_3\int_{-R}^{R}\ln |x_3|\partial_{x_3}\varphi(0,0,x_3)dx_3.
   \end{equation}

Next, write
\[
  \begin{split}
   B & =\frac{1}{\epsilon}\int_{\rho=\epsilon}G(x)\varphi(0,0,x_3)d\sigma\\
      & =\frac{1}{\epsilon}\int_{\rho=\epsilon}G(x)\varphi(0)d\sigma+\frac{1}{\epsilon}\int_{\rho=\epsilon}G(x)(\varphi(0,0,x_3)-\varphi(0))d\sigma\\
      & =: B_1+B_2.
   \end{split}
\]
We have $|B_2| \le C\int_{-R}^{R}|G(x)x_3|dx_3$
By (\ref{eqA_1}) and (\ref{eqA_G}), we have that for $|x'|=\epsilon$, $-R\le x_3\le R$, 
\[
   |G(x)x_3|\le C\frac{\rho^2x_3^2}{r^4}(|\ln\frac{\rho}{r}|^2+|\ln\frac{\rho}{r}|)\le \frac{\epsilon^2}{\epsilon^2+x_3^2}(|\ln\frac{\epsilon}{\sqrt{\epsilon^2+x_3^2}}|^2+|\ln\frac{\epsilon}{\sqrt{\epsilon^2+x_3^2}}|). 
\]
So $\lim_{\epsilon\to 0}|G(x)x_3|=0$ a.e. $x_3\in [-R, R]$,  and $|G(x)x_3|\le C$ for $-R\le x_3\le R$. By the dominated convergence theorem, we have 
\begin{equation}\label{eqA_B2}
  \lim_{\epsilon\to 0}B_2=0.
  \end{equation}

%%%%%%%%%%%%%%%%%%%
%Let $\delta=\cos\frac{R}{\sqrt{\epsilon^2+R^2}}$, we have $0<\delta<1$. On $\{\rho=\epsilon\}$, $r=\sqrt{\epsilon^2+x_3^2}$, therefore $y=\cos\theta=\frac{x_3}{r}=\frac{x_3}{\sqrt{\epsilon^2+x_3^2}}$, we have 
%\begin{equation}\label{eqA_dy}
%   dy=\frac{\epsilon^2}{(\epsilon^2+x_3^2)^{\frac{3}{2}}}dx_3, \textrm{ so }dx_3=\frac{r^3}{\epsilon}dy. 
%\end{equation}
%%So 
%%\[
%%   dx_3=\frac{r^3}{\epsilon}dy
%%\]
%We also have $\sqrt{1-y^2}=\sin\theta=\frac{\epsilon}{r}$. 
%By  (\ref{eqA_1}), (\ref{eqA_G}) and (\ref{eqA_dy}), we have 
%\[
%   \begin{split}
%      |B_2| & \le C\int_{-R}^{R}|G(x)x_3|dx_3\\
%               %& \le C\int_{-R}^{R}\frac{1}{r}\left(\sin\theta^2\cos\theta(|U_{\theta}'|^2+|U'_{\theta}|+4|c_3|)+|\cos\theta (U^2_{\theta}+\cos\theta U_{\theta})| \right)|x_3|dx_3\\
%               & =C\int_{-\delta}^{\delta}\frac{1}{r}\left(\frac{\epsilon^2}{r^2}|y|(|U_{\theta}'|^2+2|U'_{\theta}|+4|c_3|)+|y (U^2_{\theta}+yU_{\theta})| \right)r|y|\frac{r^3}{\epsilon^2}dy\\
%               & =C\epsilon\int_{-\delta}^{\delta}\left(\frac{r}{\epsilon}y^2(|U_{\theta}'|^2+2|U'_{\theta}|+4|c_3|)+ y^2\frac{r^3}{\epsilon^3} (U^2_{\theta}+yU_{\theta})| \right)dy\\
%               & \le C\epsilon\int_{-\delta}^{\delta}\left(\frac{y^2}{\sqrt{1-y^2}}(|\ln(1-y^2)|^2+C)+ \frac{y^2}{(1-y^2)^{\frac{3}{2}}}((1-y^2)|(\ln(1+y)|+|\ln(1-y)|+C)) \right)dy\\
%               & \le C\epsilon\int_{-1}^{1}\frac{y^2}{\sqrt{1-y^2}}(|\ln(1-y^2)|^2+C)dy\\
%               & \le C \epsilon
%   \end{split}
%\]
%So 
%\begin{equation}\label{eqA_B2}
%  \lim_{\epsilon\to 0}B_2=0.
%  \end{equation}

Next, let 
$
  b_{\epsilon}=\frac{1}{\epsilon}\int_{\rho=\epsilon}G(x)dx.
$ 
We have $B_1=b_{\epsilon}\varphi(0)$.  Let $\delta=R/\sqrt{\epsilon^2+R^2}$, %\frac{R}{\sqrt{\epsilon^2+R^2}}$,
 we have $0<\delta<1$. On $\{\rho=\epsilon\}$, $r=\sqrt{\epsilon^2+x_3^2}$, therefore $y=\cos\theta=x_3/r=x_3/\sqrt{\epsilon^2+R^2}$, we have %\frac{x_3}{r}=\frac{x_3}{\sqrt{\epsilon^2+x_3^2}}$, we have 
\begin{equation}\label{eqA_dy}
   dy=\frac{\epsilon^2}{(\epsilon^2+x_3^2)^{\frac{3}{2}}}dx_3, \textrm{ so }dx_3=\frac{r^3}{\epsilon^2}dy. 
\end{equation}
We also have $\sqrt{1-y^2}=\sin\theta=\epsilon/r$. %\frac{\epsilon}{r}$. 
By  (\ref{eqA_1}), (\ref{eqA_G}) and (\ref{eqA_dy}), we have %By computation,
\[
  \begin{split}
   b_{\epsilon} %& =\int_{-R}^{R}\frac{1}{r}\left(\sin\theta^2\cos\theta(|U_{\theta}'|^2+U'_{\theta}+4c_3)-\cos\theta (U^2_{\theta}+\cos\theta U_{\theta}) \right)dx_3\\
                       & =\int_{-\delta}^{\delta}\left(\frac{\epsilon^2 y}{r^3}(|U_{\theta}'|^2+2U'_{\theta}+4c_3)-\frac{y}{r}(U^2_{\theta}+y U_{\theta}) \right)\frac{r^3}{\epsilon^2}dy\\
                      % & =\int_{-\delta}^{\delta}\left(y(|U_{\theta}'|^2+2U'_{\theta}+4c_3)-\frac{yr^2}{\epsilon^2}(U^2_{\theta}+yU_{\theta}) \right)dy\\
                       & =\int_{-\delta}^{\delta}\left(y(|U_{\theta}'|^2+2U'_{\theta})-\frac{y}{1-y^2}(U^2_{\theta}+y U_{\theta}) \right)dy\\
                       & =2\delta(U_{\theta}(\delta)+U_{\theta}(-\delta))+\int_{-\delta}^{\delta}\left(y|U_{\theta}'|^2-\frac{2-y^2}{1-y^2}U_{\theta}-\frac{y}{1-y^2}U_{\theta}^2 \right)dy. 
   \end{split}
\] 
%By (\ref{eqS_1}), 
%\[
%   \frac{1}{1-y^2}(\frac{1}{2}U_{\theta}^2+2yU_{\theta})=c_3-U_{\theta}'.
%\]
%So 
%\[
%  \begin{split}
%   b_{\epsilon} & =\int_{-\delta}^{\delta}\left(y(|U_{\theta}'|^2+\frac{5}{2}U'_{\theta})-c_3x-\frac{y}{2(1-y^2)}U_{\theta}^2 \right)dx\\
%                      & =\int_{-\delta}^{\delta}\left(y|U_{\theta}'|^2+\frac{5}{2}yU'_{\theta}-\frac{x}{2(1-y^2)}U_{\theta}^2 \right)dy\\
%                      & =\frac{5}{2}\delta(U_{\theta}(\delta)+U_{\theta}(-\delta))+\int_{-\delta}^{\delta}\left(y|U_{\theta}'|^2-\frac{5}{2}U_{\theta}-\frac{y}{2(1-y^2)}U_{\theta}^2 \right)dy
%   \end{split}
%\]
As $\epsilon\to 0$, $\delta\to 1$, so $\delta(U_{\theta}(\delta)+U_{\theta}(-\delta))\to 2(U_{\theta}(1)+U_{\theta}(-1))=0$. 
By (\ref{eqA_1}), 
%\[
%   \int_{0}^{\delta}\frac{y}{2(1-y^2)}U_{\theta}^2dy \le C\int_{0}^{1}\frac{y}{2(1-y^2)}(1-y)^2(\ln(1-y))^2dy\le C
%\]
%By monotone convergence theorem, $\lim_{\epsilon\to 0}\int_{0}^{\delta}\frac{y}{2(1-y^2)}U_{\theta}^2dy$ exists and is finite. Similarly, $\lim_{\epsilon\to 0}\int_{-\delta}^{0}\frac{y}{2(1-y^2)}U_{\theta}^2dy$ exists and is finite. 
%
%
%
%By (\ref{eqA_1}), 
%\[
%   \int_{0}^{\delta}y|U_{\theta}'|^2 dy \le C\int_{0}^{1}y|\ln(1-y)|^2dy\le C,
%\]
%By the monotone convergence theorem, $\lim_{\epsilon\to 0}\int_{0}^{\delta}y|U_{\theta}'|^2 dy$ exists and is finite. Similarly,  $\lim_{\epsilon\to 0}\int_{-\delta}^{0}y|U_{\theta}'|^2 dy$ exists and is finite. 
%
%Since $U_{\theta}\in C[-1,1]$, we have 
%$
%   \lim_{\epsilon\to 0}\int_{0}^{\delta}2U_{\theta}dy
%$ exists and is finite. 
% So $b:=\lim_{\epsilon\to 0}b_{\epsilon}$ exists, finite, and 
 \begin{equation*}%\label{eqA_b}
   b:=\lim_{\epsilon\to 0}b_{\epsilon}=\int_{-1}^{1}\left(y|U_{\theta}'|^2-\frac{2-y^2}{1-y^2}U_{\theta}-\frac{y}{1-y^2}U_{\theta}^2 \right)dy.
 \end{equation*}
 
 Recall $B_1=b_{\epsilon}\varphi(0)$, we have 
 \begin{equation}\label{eqA_B1}
    \lim_{\epsilon\to 0} B_1=b\varphi(0). 
    \end{equation}
Lemma \ref{lemA_6} follows from  (\ref{eqA_A}), (\ref{eqA_B2}) and (\ref{eqA_B1}). 
\end{proof}
Proposition \ref{prop_force} follows from (\ref{eqF}), (\ref{eqA_5}), (\ref{eqA_4}), Lemma \ref{lemA_5} and Lemma \ref{lemA_6}. 
\qed

\section{Proof of Theorem \ref{thmH}}\label{sec_H}
%Denote 
%\[
%  \begin{split}
%  & \dot{H}^1(\mathbb{R}^3)=\{u\in L^1_{loc}(\mathbb{R}^3)\mid \nabla u\in L^2(\mathbb{R}^3)\},\\
%  & L^p_{\sigma}(\mathbb{R}^3)=\{u\in L^p(\mathbb{R}^3) \mid \dive u=0\},\\
%  & \dot{H}^1_{\sigma}(\mathbb{R}^3)=\{u\in \dot{H}^1(\mathbb{R}^3)\mid \dive u=0\}.
%   \end{split}
%\]

%\begin{lem}\label{lemH_1}
%   Let $n\ge 2$, $1\le p<n$, $u\in C_c^1(\mathbb{R}^n)$, $\alpha p>1-n$, $(\alpha+\beta)p>-n$, then there exists some constant  $C$, depending on $p$, $\alpha$ and $\beta$, such that
%     \begin{equation}\label{eqH_1_1}
%    \||x|^{\beta}|x'|^{\alpha} u\|_{L^p(\mathbb{R}^n)}\le C\| |x|^{\beta+\alpha-\alpha'}|x'|^{\alpha'+1}\nabla u\|_{L^p(\mathbb{R}^n)},
%   \end{equation}
%   if and only if $\alpha' \le \alpha$, 
%   where $x'=(x_1,x_2,...,x_{n-1})$ and $x=(x', x_n)$.
%\end{lem}
%\begin{proof}
\noindent\emph{Proof of Theorem \ref{thmH}}: 
For convenience, let $C$ denote a constant depending only on $p,\alpha, \beta, \alpha'$ and $n$, which may vary from line to line. We first prove that if (\ref{eqH_1_1}) holds for some $C$, then $\alpha'\le \alpha$. 

%We first prove that if  (\ref{eqH_1_1}) is  true for all  $u\in C_0^1(\mathbb{R}^n)$, then $\alpha'\le \alpha$. 

%\textbf{Claim}: If $\alpha'>\alpha$, then there exist a sequence $\{u_i\}\in C^{\infty}_c(\mathbb{R}^n)$, such that
%\[
%   \lim_{i\to \infty}\frac{\||x|^{\beta}|x'|^{\alpha} u_i\|_{L^p(\mathbb{R}^n)}}{\| |x|^{\beta+\alpha-\alpha'}|x'|^{\alpha'+1}\nabla u_I\|_{L^p(\mathbb{R}^n)}}=+\infty.
%\]
Let $0<\delta<1$, $f_{\delta}(x')$ be a smooth function of $x'$, such that
\[
   f_{\delta}(x'):=\left\{
       \begin{split}
           & 1, \quad 2\delta\le |x'|\le 3\delta, \\
           & 0, \quad |x'|\le \delta \textrm{  or }|x'|\ge 4\delta,
       \end{split}
   \right.
\]
and $|\nabla_{x'} f|\le C/\delta$.  
Let $g(x_n)$ be a smooth function such that
\[
   g(x_n):=\left\{
       \begin{split}
           & 1, \quad 2\le |x_n|\le 3, \\
           & 0, \quad |x_n|\le 1 \textrm{  or }|x_n|\ge 4,
       \end{split}
   \right.
\]
and $|g'(x_n)|\le C$. Define $u_{\delta}(x):=f_{\delta}(x')g(x_n)$, then $u_{\delta}$ is in $C_0^1(\mathbb{R}^n)$. By computation, 
\begin{equation*}%\label{eqH_1_N_1}
      \||x|^{\beta}|x'|^{\alpha} u_{\delta}\|^p_{L^p(\mathbb{R}^n)} \ge \int_{2}^{3}\int_{2\delta \le |x'|\le 3\delta}|x|^{\beta p}|x'|^{\alpha p}dx'dx_n \ge \delta^{\alpha p+n-1}/C.
\end{equation*}
%\begin{equation}\label{eqH_1_N_1}
%   \begin{split}
%      \||x|^{\beta}|x'|^{\alpha} u_{\delta}\|^p_{L^p(\mathbb{R}^n)} & = \int_{2}^{3}\int_{\delta \le |x'|\le 4\delta}|x|^{\beta p}|x'|^{\alpha p}|f_{\delta}(x')|^p|g(x_n)|^pdx'dx_n\\
%                                                                                                   & \ge C\int_{\delta \le |x'|\le 4\delta}|x'|^{\alpha p}|f_{\delta}(x')|^pdx'\\
%                                                                                                   & \ge C\int_{2\delta \le |x'|\le 3\delta}|x'|^{\alpha p}dx'\\
%                                                                                                   & \ge C\delta^{\alpha p+n-1}.
%   \end{split}
%\end{equation}
On the other hand, since $\delta\le 1$, 
\begin{equation*}%\label{eqH_1_N_2}
   \begin{split}
  &  \| |x|^{\beta+\alpha-\alpha'}|x'|^{\alpha'+1}\nabla u_{\delta}\|^p_{L^p(\mathbb{R}^n)}\\
   & \le \int_{1}^{4}\int_{\delta \le |x'|\le 4\delta}|x|^{(\beta+\alpha-\alpha') p}|x'|^{(\alpha'+1) p}(|\nabla_{x'}f_{\delta}(x')|^p|g(x_n)|^p
    +|f_{\delta}(x')|^p|g'(x_n)|^p)dx'dx_n\\
                                                                                                                           & \le  C\int_{\delta \le |x'|\le 4\delta} |x'|^{(\alpha'+1) p}(|\nabla_{x'}f_{\delta}(x')|^p+|f_{\delta}(x')|^p)dx'\\
                                                                                                                           & \le C\int_{\delta \le |x'|\le 4\delta} |x'|^{(\alpha'+1) p}\delta^{-p}dx'
                                                                                                                            \le C\delta^{\alpha' p+n-1}.
   \end{split}
\end{equation*}
Since $u_{\delta}$ satisfies (\ref{eqH_1_1}), we have $\delta^{\alpha p+n-1}\le C\delta^{\alpha' p+n-1} $
%then by (\ref{eqH_1_N_1}) and (\ref{eqH_1_N_2}), we have \[
%   \delta^{\alpha p+n-1}\le C\delta^{\alpha' p+n-1}
%\]
for any $0<\delta<1$, therefore $\alpha'\le \alpha$.

Next, we prove (\ref{eqH_1_1}) for $\alpha'\le \alpha$. Since  $|x'|\le |x|$, we only need to prove it for $\alpha'=\alpha$, i.e. 
%Next, we prove that if $\alpha'\le \alpha$, then any  $u\in C_0^1(\mathbb{R}^n)$ satisfies (\ref{eqH_1_1}).   
   %Since $|x'|\le |x|$ for all $x$, $|x|^{\beta}|x'|^{\alpha+1} \le |x|^{\beta+\alpha-\alpha'}|x'|^{\alpha'+1}$ for all $\alpha'\le \alpha$. So we only need to prove (\ref{eqH_1_1}) for $\alpha'=\alpha$, i.e. 
   \begin{equation}\label{eqH_1_1'}
       \||x|^{\beta}|x'|^{\alpha} u\|_{L^p(\mathbb{R}^n)}\le C\| |x|^{\beta}|x'|^{\alpha+1}\nabla u\|_{L^p(\mathbb{R}^n)}.
   \end{equation} 
   We introduce the spherical coordinates in $\mathbb{R}^n$. Let $r>0$, $\theta_1, ...,\theta_{n-2}\in [0,\pi]$ and $\theta_{n-1}\in [0,2\pi]$. Denote 
\[
  \begin{split}
     & x_1=r\sin\theta_1\sin\theta_2\cdots \sin\theta_{n-2}\sin\theta_{n-1},\\
     & x_2=r\sin\theta_1\sin\theta_2\cdots \sin\theta_{n-2}\cos\theta_{n-1},\\
     & x_3=r\sin\theta_1\sin\theta_2\cdots \sin\theta_{n-3}\cos\theta_{n-2},\\
     & \cdots\\
     & x_{n-1}=r\sin\theta_1\cos\theta_2,\\
     & x_n=r\cos\theta_1.
  \end{split}
\]
Then $|x'|=r\sin\theta_1$ and 
$
   dx=r^{n-1}\sin^{n-2}\theta_1\sin^{n-3}\theta_2\cdots\sin\theta_{n-2}drd\theta_1\cdots d\theta_{n-1}.
$
Let $\omega=(\theta_1,...,\theta_{n-1})$, $\omega'=(\theta_2,...,\theta_{n-1})$, $\Omega=\sin^{n-2}\theta_1\cdots\sin\theta_{n-2}$, and $\Omega'=\sin^{n-3}\theta_2\cdots\sin\theta_{n-2}$, $E=\{\omega'\mid 0\le \theta_i\le \pi, 2\le i\le n-2, 0\le \theta_{n-1}\le 2\pi\}$.  Denote $d\omega=d\theta_1\cdots d\theta_{n-1}$ and $d\omega'=d\theta_2\cdots d\theta_{n-1}$. 
%Denote $x'=(x_1,...,x_{n-1})$, then $|x'|=r\sin\theta_1$. 
 %Denote $\omega=(\theta_1,...,\theta_{n-1})$, $d\omega=d\theta_1\cdots d\theta_{n-1}$, $\omega'=(\theta_2,...,\theta_{n-1})$, $d\omega'=d\theta_2\cdots d\theta_{n-1}$,  $\Omega=\sin^{n-2}\theta_1\cdots\sin\theta_{n-2}$, and $\Omega'=\sin^{n-3}\theta_2\cdots\sin\theta_{n-2}$. 
% Let  $\Omega=\sin^{n-2}\theta_1\sin^{n-3}\theta_2\cdots\sin\theta_{n-2}$, and $\Omega'=\sin^{n-3}\theta_2\cdots\sin\theta_{n-2}$. 
We can express 
\[
   \int_{\mathbb{R}^n}(|x|^{\beta}|x'|^{\alpha}|u|)^pdx=\int_{\mathbb{R}^n}r^{(\alpha+\beta)p+n-1}|\sin\theta_1|^{\alpha p+n-2}|u|^p\Omega'drd\omega
\] 
By assumption, $\lambda:=(\alpha+\beta)p+n>0$. For each fixed $\omega\in [0,\pi]^{n-2}\times[0,2\pi]$,  let $\hat{u}(s):=u(s^{1/\lambda}, \omega)$, it is well-known that 
\[
   \int_{0}^{\infty}|\hat{u}(s)|^{p}ds\le C(p)\int_{0}^{\infty}|\hat{u}'(s)|^{p}s^pds.
\]
Namely,
%
%using the fact that $(\alpha+\beta)p>-n$,
%\[
%   \begin{split}
%     &  \int_{0}^{\infty}|u(r,\omega)|^p r^{(\alpha+\beta)p+n-1}dr\\
%      & =-\int_{0}^{\infty}\{r}^{\infty}\partial_s|u(s,\omega)|^p r^{(\alpha+\beta)p+n-1} dsdr
%                                           =-\int_{0}^{\infty}\int_{0}^{s}p|u|^{p-2}u\cdot\partial_s ur^{(\alpha+\beta)p+n-1}drds\\
%                                          & =-\frac{p}{(\alpha+\beta)p+n}\int_{0}^{\infty}|u|^{p-2}u\cdot\partial_s u s^{(\alpha+\beta)p+n}ds\\
%                                          & \le \frac{p}{(\alpha+\beta)p+n}\left(\int_{0}^{\infty}|u|^pr^{(\alpha+\beta)p+n-1}dr\right)^{1-\frac{1}{p}}\left(\int_{0}^{\infty}|\partial_r u|^pr^{(\alpha+\beta+1)p+n-1}dr\right)^{\frac{1}{p}}
%   \end{split}
%\]
%So we have 
\[
    \int_{0}^{\infty}|u(r,\omega)|^p r^{(\alpha+\beta)p+n-1}dr\le C\int_{0}^{\infty}|\partial_r u|^pr^{(\alpha+\beta+1)p+n-1}dr, \quad \forall \omega\in [0,\pi]^{n-2}\times[0,2\pi].
\]
Let $0<\epsilon<\pi/4$ be fixed, the constant $C$ also depends on $\epsilon$. By the above we have 
%\[
%     \int_{\epsilon}^{\pi-\epsilon}\int_{0}^{\infty}|u|^p r^{(\alpha+\beta)p+n-1}drd\theta_1 \le C \int_{\epsilon}^{\pi-\epsilon}\int_{0}^{\infty}|\partial_r u|^pr^{(\alpha+\beta+1)p+n-1}drd\theta_1. 
%\]
%Multiply both sides by $\Omega'$ and take integral with respect to $\omega'$, we have 
%\begin{equation}\label{eqH_1_2}
%    \int_{\{\epsilon<\theta_1<\pi-\epsilon\}}|u|^p r^{(\alpha+\beta)p+n-1}\Omega' dr d\omega'd\theta_1\le C \int_{\{\epsilon<\theta_1<\pi-\epsilon\}}|\nabla u|^p r^{(\alpha+\beta+1)p+n-1}\Omega' dr d\omega'd\theta_1. 
%\end{equation}
\begin{equation}\label{eqH_1_2}
    \int_{\epsilon}^{\pi-\epsilon}\int_E\int_{0}^{\infty}|u|^p r^{(\alpha+\beta)p+n-1}\Omega' drd\omega' d\theta_1\le C\int_{\epsilon}^{\pi-\epsilon} \int_E\int_{0}^{\infty}|\nabla u|^p r^{(\alpha+\beta+1)p+n-1}\Omega' dr d\omega'd\theta_1. 
\end{equation}
Similarly, we have 
\[
   \int_{\epsilon}^{2\epsilon}\int_E\int_{0}^{\infty}|u|^p r^{(\alpha+\beta)p+n-1}\Omega' dr d\omega'd\theta_1\le C  \int_{\epsilon}^{2\epsilon}\int_E\int_{0}^{\infty}|\nabla u|^p r^{(\alpha+\beta+1)p+n-1}\Omega' dr d\omega'd\theta_1.
\]
So there exists some $\bar{\theta}_1\in [\epsilon, 2\epsilon]$, such that 
\begin{equation}\label{eqH_3_4}
     \int_E\int_{0}^{\infty}|u(r,\bar{\theta}_1, \omega')|^pr^{(\alpha+\beta)p+n-1}\Omega' dr d\omega' \le C \int_{\epsilon}^{2\epsilon}\int_E\int_{0}^{\infty}|\partial_r u|^pr^{(\alpha+\beta+1)p+n-1}\Omega' dr d\omega'd\theta_1. 
\end{equation}
%Let $E=\{\omega'\mid 0 \theta_i\le \pi, i=2,...,n-2, 0\le \theta_{n-1}\le 2\pi\}$
%where $\{\theta_1=\bar{\theta}_1\}=\{(r,\theta)\mid r>0, \theta_1=\bar{\theta}_1, 0\le \theta_i\le \pi, i=1,...,n-2, 0\le \theta_{n-1}\le 2\pi\}$,  $\{a<\theta_1<b\}:=\{(r,\theta)\mid r>0, a<\theta_1<b, 0\le \theta_i\le \pi, i=1,...,n-2, 0\le \theta_{n-1}\le 2\pi\}$ for any $a<b$. 
Notice for $\theta_1\in [0, \frac{\pi}{2}]$, $\theta_1\le \sin\theta_1\le 2\theta_1$. By computation, using $\alpha p+n>1$, for every fixed $r$ and $\omega'$, 
\[
   \begin{split}
    &  \int_{0}^{\bar{\theta}_1}|u(r,\theta_1, \omega')|^p|\sin\theta_1|^{\alpha p+n-2}d\theta_1 \\
     %\le C\int_{0}^{\bar{\theta}_1}|u|^p\theta_1^{\alpha p+n-2}d\theta_1\\
                                                            & =\int_{0}^{\bar{\theta}_1}|\sin\theta_1|^{\alpha p+n-2}\left(|u(r, \bar{\theta}_1, \omega')|^p-\int_{\theta_1}^{\bar{\theta}_1}\partial_t |u(r,t, \omega')|^{p}dt\right)d\theta_1\\
                                                            & \le C\int_{0}^{\bar{\theta}_1}\theta_1^{\alpha p+n-2}|u(r, \bar{\theta}_1, \omega')|^pd\theta_1+C\int_{0}^{\bar{\theta}_1}|u(r, t, \omega')|^{p-1}|\partial_t u(r,t, \omega'|)\int_{0}^{t}\theta_1^{\alpha p+n-2}d\theta_1 dt\\
                                                           % & =\frac{C}{\alpha p+n-1}|u(r, \bar{\theta}_1, \omega')|^p\bar{\theta}_1^{\alpha p+n-1}-\frac{C}{\alpha p+n-1}\int_{0}^{\bar{\theta}_1}t^{\alpha p+n-1}p|u(r, t, \omega')|^{p-2}u(r, t, \omega')\partial_t u(r,t, \omega')dt\\
                                                            & \le C|u(r, \bar{\theta}_1, \omega')|^p+C\int_{0}^{\bar{\theta}_1}|\sin t|^{\alpha p+n-1} |u(r,t,\omega')|^{p-1}|\partial_{t} u|dt\\
                                                            & \le C |u(r, \bar{\theta}_1, \omega')|^p+\frac{1}{2}\int_{0}^{\bar{\theta}_1}|u|^p |\sin t|^{\alpha p+n-2}dt
                                                             +C\int_{0}^{\bar{\theta}_1}|\sin t|^{(\alpha +1)p+n-2} |\partial_{t} u|^pd\theta_1.
   \end{split}
\]
Thus 
\[
     \int_{0}^{\bar{\theta}_1}|u(r,\theta_1, \omega')|^p|\sin\theta_1|^{\alpha p+n-2}d\theta_1
     \le C|u(r, \bar{\theta}_1, \omega')|^p+C\int_{0}^{\bar{\theta}_1}|\sin\theta_1|^{(\alpha+1)p+n-2} |\partial_{\theta_1} u(r,\theta_1, \omega')|^pd\theta_1.
\]
Multiply both sides of the above by $r^{(\alpha+\beta)p+n-1}\Omega'$, and take integral with respect to $r$ and $\omega'$. By (\ref{eqH_3_4}), we have 
\begin{equation}\label{eqH_3_5}
   \begin{split}
    & \int_{0}^{\epsilon}\int_E\int_{0}^{\infty}|u|^pr^{(\alpha+\beta)p+n-1}|\sin\theta_1|^{\alpha p+n-2}\Omega' dr d\omega'd\theta_1 \\
    &   \le \int_{0}^{\bar{\theta}_1} \int_E\int_{0}^{\infty}|u|^pr^{(\alpha+\beta)p+n-1}|\sin\theta_1|^{\alpha p+n-2}\Omega' dr d\omega'd\theta_1\\
                                                                              & \le C \int_E\int_{0}^{\infty}|u(r,\bar{\theta}_1,\omega')|^pr^{(\alpha+\beta)p+n-1}\Omega' dr d\omega'\\
                                                                              & +C \int_{0}^{\bar{\theta}_1}\int_E\int_{0}^{\infty}r^{(\alpha+\beta)p+n-1}|\sin\theta_1|^{(\alpha+1)p+n-2} |\partial_{\theta_1} u|^p\Omega' dr d\omega'd\theta_1\\
                                                                              & \le C \int_{\epsilon}^{2\epsilon}\int_E\int_{0}^{\infty}|\partial_r u|^pr^{(\alpha+\beta+1)p+n-1}\Omega' dr d\omega'd\theta_1\\
                                                                               &+C \int_{0}^{\bar{\theta}_1}\int_E\int_{0}^{\infty}r^{(\alpha+\beta)p+n-1}|\sin\theta_1|^{(\alpha+1)p+n-2} |\partial_{\theta_1} u|^p\Omega' dr d\omega'd\theta_1\\
                                                                              & \le C \int_{0}^{2\epsilon}\int_E\int_{0}^{\infty}|\nabla u|^pr^{(\alpha+\beta+1)p+n-1}|\sin\theta_1|^{(\alpha+1)p+n-2}\Omega' dr d\omega'd\theta_1
                                                                               \end{split}
\end{equation}

Similarly, we have 
\begin{equation}\label{eqH_3_6}
\begin{split}
 &  \int_{\pi-\epsilon}^{\pi} \int_E\int_{0}^{\infty} |u|^pr^{(\alpha+\beta)p+n-1}|\sin\theta_1|^{\alpha p+n-2}\Omega' dr d\omega'd\theta_1\\
 & \le C\int_{\pi-2\epsilon}^{\pi}\int_E\int_{0}^{\infty}|\nabla u|^pr^{(\alpha+\beta+1)p+n-1}|\sin\theta_1|^{(\alpha+1)p+n-2}\Omega' dr d\omega'd\theta_1
  \end{split}
\end{equation}

By (\ref{eqH_1_2}), (\ref{eqH_3_5}) and (\ref{eqH_3_6}), we have 
\[
  \begin{split}
   &  \int_{\mathbb{R}^n}|u|^pr^{(\alpha+\beta)p+n-1}|\sin\theta_1|^{\alpha p+n-2}\Omega' dr d\omega \\
   & \le C\int_{\mathbb{R}^n}|\nabla u|^pr^{(\alpha+\beta+1)p+n-1}|\sin\theta_1|^{(\alpha+1)p+n-2}\Omega' dr d\omega,
   \end{split}
\]
which is equivalent to (\ref{eqH_1_1'}). The theorem is proved.
%i.e.
%\[
% \int_{\mathbb{R}^n}|u|^pr^{(\alpha+\beta)p}|\sin\theta_1|^{\alpha p}r^{n-1}\Omega dr d\omega \le C\int_{\mathbb{R}^n}|\nabla u|^pr^{(\alpha+\beta+1)p}|\sin\theta_1|^{(\alpha+1)p}r^{n-1}\Omega dr d\omega,
%\]
%i.e. 
%\[
%    \int_{\mathbb{R}^n}|u|^p|x|^{\beta p}|x'|^{\alpha p}dx \le C\int_{\mathbb{R}^n}|\nabla u|^p|x|^{\beta p}|x'|^{(\alpha+1)p}dx.
%    \]
%    The lemma is proved.
%\end{proof}
\qed

\begin{cor}\label{corH_1}
 Let $K$ be a compact subset of $M$, $(c,\gamma)\in K$. Then there exist some positive constant $C$, depending only on $K$, such that for any $w\in \dot{H}^1(\mathbb{R}^3)$, 
  \begin{equation*}%\label{eqcorH_1_1}
     \int_{\mathbb{R}^3}|w|^2|\nabla u^{c,\gamma}| dx+\int_{\mathbb{R}^3}|w|^2|u^{c,\gamma}|^2dx \le C(|c|+|\gamma|)\|\nabla w\|_{L^2}^2.
  \end{equation*}
\end{cor}
\begin{proof}
By Corollary \ref{corS_2}, we have 
   \[
     |u^{c,\gamma}|\le \frac{C(|c|+|\gamma|)}{\sqrt{|x||x'|}}, \quad  |\nabla u^{c,\gamma}|\le \frac{C(|c|+|\gamma|)}{|x||x'|}.
   \]
   By Theorem \ref{thmH} with $\alpha=\beta=-\frac{1}{2}$, $p=2$ and $n=3$, we have
 \begin{equation*}%\label{eqcorH_1_3}
   \begin{split}
      & \int_{\mathbb{R}^3}|w|^2|u^{c,\gamma}|^2dx+\int_{\mathbb{R}^3}|w|^2|\nabla u^{c,\gamma}| dx\\
     & \le  C(|c|+|\gamma|)\int_{\mathbb{R}^3}\frac{|w|^2}{|x||x'|}dx 
       \le C(|c|+|\gamma|)\int_{\mathbb{R}^3}|x|^{-1}|x'||\nabla w|^2dx 
        \le C(|c|+|\gamma|)\|\nabla w\|_{L^2}^2. 
      \end{split}
         \end{equation*}
%%%%%%%%%%%%%%%%%%%%%%%%%%%%%%%%%%%%%%%%%%%%%%%%
%   By Corollary \ref{corS_2}, we have 
%   \[
%      |\nabla u^{c,\gamma}|\le \frac{C(|c|+|\gamma|)}{|x||x'|}.
%   \]
%   Let $\alpha=\beta=-\frac{1}{2}$, $p=2$ and $n=3$, we have $\alpha p>-2=1-n$ and $(\alpha+\beta)p>-3=-n$. By the density of $C^{\infty}_c(\mathbb{R}^3)$ in $\dot{H}^1(\mathbb{R}^3)$, we apply Lemma \ref{lemH_1} on $w$ and $v$ with $\alpha=\beta=-\frac{1}{2}$, $p=2$ and $n=3$, we have
%   \begin{equation}\label{eqcorH_1_3}
%      \int_{\mathbb{R}^3}|w|^2|\nabla u^{c,\gamma}| dx\le  \int_{\mathbb{R}^3}\frac{|w|^2}{|x||x'|}dx\le \int_{\mathbb{R}^3}|x|^{-1}|x'||\nabla w|^2dx\le \|\nabla w\|_{L^2}^2. 
%         \end{equation}
%    Next, by Corollary \ref{corS_2} and (\ref{eqcorH_1_3}), we have 
%   \[
%     \begin{split}
%       \int_{\mathbb{R}^3}|w|^2|u^{c,\gamma}|^2dx                                                                                                       & \le C(|c|+|\gamma|)\int_{\mathbb{R}^3}\frac{|w|^2}{|x|^2}(|\ln(|x'|/|x|)|+1)^2 dx\\
%                                                                                                      & \le C(|c|+|\gamma|)\int_{\mathbb{R}^3}\frac{|w|^2}{|x||x'|} dx\\
%                                                                                                      & \le C(|c|+|\gamma|)\|\nabla w\|^2_2
%       \end{split}
%   \]
%   By the above and (\ref{eqcorH_1_3}),  %(\ref{eqcorH_1_1}) is proved. 
%   the lemma is proved.
\end{proof}
Notice (\ref{eqEst_1}) follows from Corollary \ref{corH_1}.

\FloatBarrier

\end{document}